\documentclass[a4paper,11pt]{article}

\usepackage[margin=1.0in]{geometry}

\usepackage{amsmath}
\usepackage{algorithmic}
\usepackage{algorithm}
\usepackage{amsthm}
\usepackage{amsfonts}
\usepackage{comment}
\usepackage{graphicx}
\usepackage{subcaption}
\usepackage{hyperref}
\usepackage{color}

\newtheorem{theorem} {Theorem}
\newtheorem{lemma} {Lemma}
\newtheorem{definition} {Definition}

\newtheorem{observation} {Observation}

\def\x{{\mathbf{x}}}

\def\v{{\mathbf{v}}}

\def\w{{\mathbf{w}}}
\def\y{{\mathbf{y}}}

\def\X{{\mathbf{X}}}
\def\Y{{\mathbf{Y}}}
\def\A{{\mathbf{A}}}
\def\M{{\mathbf{M}}}
\def\I{{\mathbf{I}}}
\def\B{{\mathbf{B}}}

\def\V{{\mathbf{V}}}
\def\Z{{\mathbf{Z}}}
\def\W{{\mathbf{W}}}

\def\P{{\mathbf{P}}}

\def\EV{{\mathbf{EV}}}

\def\matE{{\mathbf{E}}}

\newcommand{\mK}{\mathcal{K}}
\newcommand{\mS}{\mathcal{S}}
\newcommand{\ball}{\mathcal{NB}}
\newcommand{\mbS}{\mathbb{S}}

\newcommand{\E}{\mathbb{E}}

\newcommand{\gap}{\textrm{gap}}
\newcommand{\nnz}{\textrm{nnz}}
\newcommand{\trace}{\textrm{Tr}}
\newcommand{\rank}{\textrm{rank}}
\newcommand{\reals}{\mathbb{R}}

\title{Faster Projection-free Convex Optimization over the Spectrahedron}
\date{}
\author{Dan Garber \\
Toyota Technological Institute at Chicago \\ 
\small{dgarber@ttic.edu}
}

\begin{document}

\maketitle
 
\begin{abstract}
Minimizing a convex function over the spectrahedron, i.e., the set of all $d\times d$ positive semidefinite matrices with unit trace, is an important optimization task with many applications in optimization, machine learning, and signal processing, the most notable one probably being \textit{matrix completion}. Unfortunately, it is also notoriously difficult to solve in large-scale since standard techniques require to compute expensive matrix decompositions on each iteration. An alternative, is the conditional gradient method (aka Frank-Wolfe algorithm) that regained much interest in recent years, mostly due to its application to this specific setting. The key benefit of the CG method is that it avoids expensive matrix decompositions all together, and simply requires a single eigenvector computation per iteration, which is much more efficient. On the downside, the CG method, in general, converges with an inferior rate. The error for minimizing a $\beta$-smooth function after $t$ iterations scales like $\beta/t$. This convergence rate does not improve even if the function is also strongly convex.

In this work we present a modification of the CG method tailored for convex optimization over the spectrahedron. The per-iteration complexity of the method is essentially identical to that of the standard CG method: only a single eigenvecor computation is required. For minimizing an $\alpha$-strongly convex and $\beta$-smooth function, the \textit{expected} approximation error of the method after $t$ iterations is:
\begin{eqnarray*}
O\left({\min\{\frac{\beta{}}{t} ,\left({\frac{\beta\sqrt{\rank(\X^*)}}{\alpha^{1/4}t}}\right)^{4/3}, \left({\frac{\beta}{\sqrt{\alpha}\lambda_{\min}(\X^*)t}}\right)^{2}\}}\right),
\end{eqnarray*}
where $\rank(\X^*), \lambda_{\min}(\X^*)$ are the rank of the optimal solution, and smallest non-zero eigenvalue, respectively. Beyond the significant improvement in convergence rate,  it also follows that when the optimum is low-rank, our method provides better accuracy-rank tradeoff than the standard CG method.

To the best of our knowledge, this is the first result that attains provably faster convergence rates for a CG variant for optimization over the spectrahedron. We also present encouraging preliminary empirical evidence, that shows that our approach may improve also in practice over previous projection-free methods.
\end{abstract}

\section{Introduction}

Minimizing a convex function over the set of positive semidefinite matrices with unit trace, aka the spectrahedron, is an important optimization task which lies at the heart of many optimization, machine learning, and signal processing tasks such as matrix completion \cite{Candes09, Recht11, Jaggi10}, metric learning \cite{Xing03,Weinberger05,Ying12}, kernel matrix learning \cite{Lanckriet04,Gonen11}, multiclass classification \cite{Dudik12a,Zhang12,Hazan16}, and more.

Since modern applications are mostly of very large scale, first-order methods are the obvious choice to deal with this optimization problem. However, even these are notoriously difficult to apply, since most of the popular gradient schemes require the computation of an orthogonal projection on each iteration to enforce feasiblity, which for the spectraheron, amounts to computing a full eigen-decomposition of a real symmetric matrix. Such a decomposition requires $O(d^3)$ arithmetic operations for a $d\times d$ matrix, and thus is prohibitive for high-dimensional problems. An alternative is to use first-order methods that do not require expensive decompositions, but rely only on computationally-cheap leading eigenvector computations. These methods are mostly based on the conditional gradient method, also known as the Frank-Wolfe algorithm \cite{FrankWolfe, Jaggi13b}, which is a generic method for constrained convex optimization given an oracle for minimizing linear functions over the feasible domain. Indeed, linear minimization over the spectrahedron amounts to a single leading eigenvector computation. While the CG method has been discovered already in the 1950's \cite{FrankWolfe, Polyak}, it has regained much interest in recent years in the machine learning and optimization communities, in particular due to its applications to semidefinite optimization and convex optimization with a nuclear norm constraint / regularization\footnote{minimizing a convex function subject to a nuclear norm constraint is efficiently reducible to the minimization of the function over the spectrahedron, as we detail in Subsection \ref{sec:nuclear2spectra}.}, e.g., \cite{Hazan08, Jaggi10, Laue12, ShalevShwartz11,Ying12, Dudik12a, Dudik12b, Hazan12, Hazan16}. This regained interest is not surprising: while a full eigen-decomposition for $d\times d$ matrix requires $O(d^3)$ arithmetic operations, leading eigenvecor computations can be carried out, roughly speaking, in worst-case time that is only linear in the number of non-zeros in the input matrix multiplied by either $\epsilon^{-1}$ for the popular Power Method or by $\epsilon^{-1/2}$ for the more efficient Lanczos method, where $\epsilon$ is the target accuracy. These running times improve exponentially to only depend on $\log(1/\epsilon)$ when the eigenvalues of the input matrix are well distributed \cite{EigenvaluesApprox}. Indeed, in several important machine learning applications, such as matrix completion, the CG method requires eigenvector computations of very sparse matrices \cite{Jaggi10}. Also, very recently, new eigenvector algorithms with significantly improved performance guarantees were introduced which are applicable for matrices with certain popular structure \cite{Garber15, Sidrord15,Shamir15}.

Because of their cheap iteration complexity, conditional gradient-based methods are also of interest in online optimization settings, such as \textit{online convex optimization} or \textit{online stochastic optimization}, in which, roughly speaking, given a continuos stream of data, one wants to incrementally update the prediction / hypothesis based on newly observed data. In these settings the time required for the optimization method to perform a single update may be a key consideration in its applicability to the problem \cite{Hazan12, Garber13, Garber13b}.

The main drawback of the CG method is that its convergence rate is, in general, inferior compared to projection-based gradient methods. The convergence rate for minimizing a smooth function, roughly speaking, scales only like $1/t$. In particular, in general, this rate does not improve, even when the function is also strongly convex. On the other hand, the convergence rate of optimal projection-based methods, such as Nesterov's accelerated gradient method, scales like $1/t^2$ for smooth functions, and can be improved exponentially to $\exp(-\Theta(t))$ when the objective is also strongly convex \cite{Nesterov13}.

Very recently, several successful attempts were made to devise natural modifications of the CG method that retain the overall low per-iteration complexity, while enjoying provably faster convergence rates, usually under a strong-convexity assumption, or a slightly weaker one. These results exhibit provably-faster rates for optimization over polyhedral sets \cite{Garber13b,Jaggi13c,Beck15} and strongly-convex sets \cite{GH15}, but do not apply to the spectrahedron. For the specific setting considered in this work, several heuristic improvements of the CG method were suggested which show promising empirical evidence, however, non of them provably improve over the rate of the standard CG method \cite{ShalevShwartz11, Laue12, Freund15}.

In this work, we present, a new non-trivial variant of the CG method, which, to the best of our knowledge, is the first one to exhibit provably faster convergence rates for optimization over the spectrahedron under standard smoothness and strong convexity assumptions. The per-iteration complexity of the method is essentially identical to that of the standard CG method in this setting, i.e., only a single leading eigenvector computation per iteration is required.

Our method is tailored for optimization over the spectrahedron, and can be seen as a certain hybridization of the standard CG method and the projected gradient method. From a high-level view, we take advantage of the fact  that solving a $\ell_2$-regularized linear problem over the set of extreme points of the spectrahedron is equivalent to linear optimization over this set, i.e., amounts to a single eigenvector computation. We then show via a novel and non-trivial analysis, that includes new decomposition concepts for positive semidefinite matrices, that such an algorithmically-cheap regularization is sufficient, in presence of strong convexity, to derive faster convergence rates.

While the combination of smoothness and strong convexity is a rare commodity, several important problems such as linear regression in the well-conditioned case, and solving undetermined linear systems (such as in the matrix completion problem), under certain conditions (see for instance  \cite{Negahban09}), exhibit such properties. Moreover, since computing the euclidean projection is a smooth and strongly convex optimization problem with respect to the $\ell_2$ norm, our method can be readily used to simulate any $\ell_2$-projection-based algorithm, replacing the projection step with only a leading eigenvector step. This approach has allowed, among other things, to apply CG-based methods to non-smooth problems, for which the standard CG method is not suitable \cite{Garber13b}, and to strike better trade-offs between the linear optimization oracle complexity and the first-order oracle complexity \cite{Lan14,Hazan16}

\begin{table*}[t!]\label{table:compare}
\begin{center}
  \begin{tabular}{| l | c | c | }
    \hline
    Method &  \#iterations to $\epsilon$ error & Iteration complexity \\ \hline
    Proj. Grad. & $\frac{\beta}{\alpha}\log(1/\epsilon)$ & $d^3$  \\ \hline
    Acc. Grad. & $\sqrt{\frac{\beta}{\alpha}}\log(1/\epsilon)$ & $d^3$   \\ \hline
    Cond. Grad. & $\beta/\epsilon$ & $\nnz(\nabla)\sqrt{\Vert{\nabla}\Vert_2}\min\{\epsilon^{-1/2}, \, \frac{\log(1/\epsilon)}{\gap(\nabla)}\}$\\ \hline
    Algorithm \ref{alg:1} & $\min\{\frac{\beta}{\epsilon}, \, \frac{\beta\sqrt{\rank(\X^*)}}{\alpha^{1/4}\epsilon^{3/4}}, \, \frac{\beta}{\sqrt{\alpha}\lambda_{\min}(\X^*)\sqrt{\epsilon}}\}$ & $(\nnz(\nabla)+d)\sqrt{\Vert{\tilde{\nabla}}\Vert_2}\min\{\epsilon^{-1/2}, \, \frac{\log(1/\epsilon)}{\sqrt{\gap(\tilde{\nabla})}}\}$ \\ \hline
  \end{tabular}
  \caption{Comparison between first-order methods for minimizing an $\alpha$-strongly convex and $\beta$-smooth function over the spectrahedron in $\reals^{d\times d}$. $\nnz(\nabla)$ is number of non-zeros in the gradient matrix in any of the algorithm's iterations, and $\gap(\nabla)$ is the difference between the smallest and second smallest eigenvalues of the gradient. For Algorithm \ref{alg:1}, we use the notation $\tilde{\nabla}$ since the gradient is perturbed with a small rank-one matrix.  The running times for the eigenvector computation are based on the Lanczos method \cite{EigenvaluesApprox}.} 
\end{center}
\end{table*}

\subsection{Paper organization}

The rest of this paper is organized as follows. In Section \ref{sec:prelim} we give necessary preliminaries and notation, describe the problem considered in this paper in full detail, and draw known connections to the popular problem of convex optimization under a nuclear norm constraint. In Section \ref{sec:approach} we briefly describe the conditional gradient and projected gradient methods for optimization over the spectrahedron, and present our new method, which is a certain hybridization of the two. We also state the main theorem of this paper, Theorem \ref{thm:main}, which describes the novel convergence rate of the proposed method. In Section \ref{sec:analysis} we analyze our proposed method and prove the main theorem, Theorem \ref{thm:main}. Finally, in Section \ref{sec:experiments} we present preliminary empirical evidence that shows that our method may indeed improve in practice over previous conditional gradient methods.

\section{Preliminaries and Notation}\label{sec:prelim}
Throughout this work we use boldface lowercase letters to denote vectors in $\reals^d$, e.g. $\v$, boldface uppercase letters to denote matrices, e.g. $\X$, and lightface letters to denote scalars. For vectors we let $\Vert\cdot\Vert$ denote the standard Euclidean norm, while for matrices we let $\Vert\cdot\Vert$ denote the spectral norm, $\Vert{\cdot}\Vert_F$ denote the Frobenius norm, and $\Vert\cdot\Vert_*$ denote the nuclear norm.  
We denote by $\mbS_d$ the space of $d\times d$ real symmetric matrices, and by $\mS_d$ the \textit{spectrahedron} in $\mbS_d$, i.e., 
\begin{eqnarray*}
\mS_d := \{\X \in \mbS_d \, | \, \X\succeq 0, \trace(\X) = 1\} .
\end{eqnarray*}

We let $\trace(\cdot)$ and $\rank(\cdot)$ denote the trace and rank of a given matrix in $\mbS_d$, respectively. We let $\bullet$ denote the standard inner-product for matrices. Given a matrix $\X\in\mS_d$, we let $\lambda_{\min}(\X)$ denote the smallest non-zero eigenvalue of $\X$.


Throughout this work, given a matrix $\A\in\mbS_d$, we denote by $\EV(\A)$ an eigenvector of $\A$ that corresponds to the largest (signed) eigenvalue of $\A$, i.e., $\EV(\A)\in\arg\max_{\v : \Vert{\v}\Vert = 1}\v^{\top}\A\v$. Given a scalar $\xi > 0$, we also denote by $\EV_{\xi}(\A)$ an $\xi$-approximation to the largest (in terms of eigenvalue) eigenvector of $\A$, i.e., $\EV_{\xi}(\A)$ returns a unit vector $\v$ such that $\v^{\top}\A\v \geq \lambda_{\max}(\A) - \xi$.


\begin{definition}
We say that a function $f(\X):\reals^{m\times n}\rightarrow\reals$ is $\alpha$-strongly convex w.r.t. a norm $\Vert\cdot\Vert$,
if for all $\X,\Y\in\reals^{m\times n}$ it holds that
\begin{eqnarray*}
f(\Y) \geq f(\X) + (\Y-\X)\bullet\nabla{}f(\X) + \frac{\alpha}{2}\Vert{\X-\Y}\Vert^2 .
\end{eqnarray*}
\end{definition}

\begin{definition}
We say that a function $f(\X):\reals^{m\times n}\rightarrow\reals$ is $\beta$-smooth w.r.t. a norm $\Vert\cdot\Vert$, if for all $\X,\Y\in\reals^{m\times n}$ it holds that
\begin{eqnarray*}
f(\Y) \leq f(\X) + (\Y-\X)\bullet\nabla{}f(\X) + \frac{\beta}{2}\Vert{\X-\Y}\Vert^2 .
\end{eqnarray*}
\end{definition}

The first-order optimality condition implies that for a $\alpha$-strongly convex $f$, if $\X^*$ is the unique minimizer of $f$ over a convex set $\mK\subset\reals^{m\times n}$, then for all $\X\in\mK$ it holds that
\begin{eqnarray}\label{eq:strongconvexdist}
f(\X) - f(\X^*) \geq \frac{\alpha}{2}\Vert{\X-\X^*}\Vert^2 .
\end{eqnarray}


\subsection{Problem setting}

The main focus of this work is the following optimization problem:
\begin{eqnarray}\label{eq:spectraOptProb}
\min_{\X\in\mS_d}f(\X),
\end{eqnarray}
where we assume that $f(\X)$ is both $\alpha$-strongly convex and $\beta$-smooth w.r.t. $\Vert\cdot\Vert_F$. We denote the (unique) minimizer of $f$ over $\mS_d$ by $\X^*$. 



\subsection{Convex optimization with a nuclear norm constraint}\label{sec:nuclear2spectra}
An important optimization problem highly-related to Problem \eqref{eq:spectraOptProb}, is the problem of minimizing a convex function over the set of $d_1\times d_2$ real-valued matrices with bounded nuclear norm, i.e,

\begin{eqnarray}\label{eq:nuclearOptProb} 
\min_{\Z\in\ball_{d_1,d_2}(\theta)}f(\Z).
\end{eqnarray}
Here we let $\ball_{d_1,d_2}(\theta)$ denote the nuclear-norm ball of radius  $\theta$ in $\reals^{d_1\times d_2}$, i.e.,
\begin{eqnarray*}
\ball_{d_1,d_2}(\theta) := \{\Z\in\reals^{d_1\times d_2} \, | \, \Vert{\Z}\Vert_* := \sum_{i=1}^{\min\{d_1,d_2\}}\sigma_i(\Z) \leq \theta\},
\end{eqnarray*}
where we let $\sigma(\Z)$ denote the vector of singular values of $\Z$.

Problem \eqref{eq:nuclearOptProb} could be directly formulated as convex optimization over the spectrahedron. Towards this end, consider now the following convex optimization problem:
\begin{eqnarray*}
&\min_{\X\in\mS_{d_1+d_2}}  \hat{f}(\X), & \nonumber \\
&\hat{f}(\X):=f(2\theta\cdot\M_1\X\M_2), \quad 
\M_1 := \left( \begin{array}{cc}
\I_{d_1} & \textbf{0}_{d_1\times{}d_2}  \end{array} \right), \quad
\M_1 := \left( \begin{array}{c}
\textbf{0}_{d_1\times{}d_2}  \\
\I_{d_2}  \end{array} \right) . &
\end{eqnarray*}

The following Lemma, whose proof can be found in \cite{Jaggi10}, shows the equivalence between the two problems.

\begin{lemma}
Let $\X\in\mS_{d_1+d_2}$ such that $\hat{f}(\X) - \hat{f}(\X^*) = \epsilon$, for some $\epsilon >0$, where $\X^*$ is the minimizer of $\hat{f}$ over $\mS_{d_1+d_2}$. Consider the following factorization of $\X$:
\begin{eqnarray*}
&\X & = \left( \begin{array}{cc}
\X_1 & \X_2 \\
\X_2^{\top} & \X_3  \end{array} \right),
\end{eqnarray*}
where $\X_1$ is $d_1\times d_1$, $\X_2$ is $d_1 \times d_2$, and $\X_3$ is $d_2\times d_2$. Define $\Z := 2\theta\cdot\X_2$.
Then it follows that $\Z\in\ball_{d_1,d_2}(\theta)$, and $f(\Z) - f(\Z^*) = \epsilon$, where $\Z^*$ is the minimizer of $f$ over $\ball_{d_1,d_2}(\theta)$.
\end{lemma}

\section{Our Approach}\label{sec:approach}

In order to better communicate our ideas, we begin by briefly describing the conditional gradient and projected-gradient methods, pointing out their advantages and short-comings for solving Problem \eqref{eq:spectraOptProb} in Subsection \ref{sec:approach:old}. We then present our new method which is a certain combination of ideas from both methods in Subsection \ref{sec:approach:new}.

\subsection{Conditional gradient and projected gradient descent}\label{sec:approach:old}

The standard conditional gradient algorithm is detailed below in Algorithm \ref{alg:cg}.

\begin{algorithm}
\caption{Conditional Gradient}
\label{alg:cg}
\begin{algorithmic}[1]
\STATE input: sequence of step-sizes $\{\eta_t\}_{t\geq 1}\subset[0,1]$
\STATE let $\X_1$ be an arbitrary matrix in $\mS_d$
\FOR{$t = 1...$}
\STATE $\v_t \gets \EV\left({-\nabla{}f(\X_t)}\right)$
\STATE $\X_{t+1} \gets \X_t + \eta_t(\v_t\v_t^{\top}-\X_t)$
\ENDFOR
\end{algorithmic}
\end{algorithm}

Let us denote the approximation error of Algorithm \ref{alg:cg} after performing $t$ iterations by $h_t := f(\X_t) - f(\X^*)$.

The convergence result of Algorithm \ref{alg:cg} is based on the following simple observations:
\begin{eqnarray}\label{eq:oldCGanalysis}
h_{t+1} &=& f(\X_t + \eta_t(\v_t\v_t^{\top}-\X_t)) - f(\X^*) \nonumber \\
&\leq & h_t + \eta_t(\v_t\v_t^{\top} - \X_t)\bullet \nabla{}f(\X_t) + \frac{\eta_t^2\beta}{2}\Vert{\v_t\v_t^{\top}-\X_t}\Vert_F^2 \nonumber \\
&\leq & h_t + \eta_t(\X^* - \X_t)\bullet \nabla{}f(\X_t) + \frac{\eta_t^2\beta}{2}\Vert{\v_t\v_t^{\top}-\X_t}\Vert_F^2 \nonumber \\
&\leq & (1-\eta_t)h_t  + \frac{\eta_t^2\beta}{2}\Vert{\v_t\v_t^{\top}-\X_t}\Vert_F^2,
\end{eqnarray}
where the first inequality follows from the $\beta$-smoothness of $f(\X)$, the second one follows for the optimal choice of $\v_t$, and the third one follows from convexity of $f(\X)$. Unfortunately, while we expect the error $h_t$ to rapidly converge to zero, the term $\Vert{\v_t\v_t^{\top}-\X_t}\Vert_F^2$ in Eq. \eqref{eq:oldCGanalysis}, in principal, might remain as large as the diameter of $\mS_d$, which, given a proper choice of step-size $\eta_t$, results in the well-known convergence rate of $O(\beta/t)$ \cite{Jaggi13b,Hazan08}. This consequence holds also in case $f(\X)$ is not only smooth, but also strongly-convex, see for instance Lemma 21 in \cite{JaggiThesis}.

However, in case $f$ is strongly convex, a non-trivial  modification of Algorithm \ref{alg:cg} can lead to a much faster convergence rate. In this case, it follows from Eq. \eqref{eq:strongconvexdist}, that on any iteration $t$, $\Vert{\X_t - \X^*}\Vert_F^2 \leq \frac{2}{\alpha}h_t$. Thus, if we consider replacing the choice of $\X_{t+1}$ in Algorithm \ref{alg:cg} with the following update rule:
\begin{eqnarray}\label{eq:expensiveCGrule}
\V_t \gets {\arg\min}_{\V\in\mS_d}\V \bullet \nabla{}f(\X_t) + \frac{\eta_t\beta}{2}\Vert{\V_t - \X_t}\Vert_F^2, \qquad \X_{t+1} \gets \X_t + \eta_t(\V_t - \X_t),
\end{eqnarray}
then, following basically the same steps as in Eq. \eqref{eq:oldCGanalysis}, we will have that
\begin{eqnarray}\label{eq:projectionCGanalysis}
h_{t+1} &\leq &  h_t+ \eta_t(\X^* - \X_t)\bullet \nabla{}f(\X_t) + \frac{\eta_t^2\beta}{2}\Vert{\X^*-\X_t}\Vert_F^2 \leq \left({1 - \eta_t + \frac{\eta_t^2\beta}{\alpha}}\right)h_t,
\end{eqnarray}
and thus by a proper choice of $\eta_t$, a linear convergence rate will be attained.
Of course the issue now, is that computing $\V_t$ is no longer a computationally-cheap leading eigenvalue problem (in particular $\V_t$ is not rank-one), but requires a full eigen-decomposition of $\X_t$, which is much more expensive. In fact, the update rule in Eq. \eqref{eq:expensiveCGrule} is nothing more than the projected gradient decent method, which, in spite of the linear convergence rate, is inefficient for large-scale matrix problems because of the need to compute expansive decompositions.

\subsection{A new hybrid approach: rank one-regularized conditional gradient algorithm}\label{sec:approach:new}

At the heart of our new method is the combination of ideas from both of the above approaches: on one hand, solving a certain regularized linear problem in order to avoid the shortcomings of the CG method, i.e., slow convergence rate, and on the other hand, maintaining the simple structure of a leading eigenvalue computation that avoids the shortcoming of the computationally-expensive projected-gradient method. 

Towards this end, suppose that have an explicit decomposition of the current iterate $\X_t = \sum_{i=1}^ka_i\x_i\x_i^{\top}$, where $k$ is an integer, $(a_1,a_2,...,a_k)$ is a probability distribution over $[k]$, and each $\x_i$ is a unit vector. Note in particular that the standard CG method (Algorithm \ref{alg:cg}) naturally produces such an explicit decomposition of $\X_t$. Consider now, the update rule in Eq. \eqref{eq:expensiveCGrule}, but with the additional restriction that $\V_t$ is rank one, i.e,
\begin{eqnarray}\label{eq:towardsNewUpdate}
\V_t \gets {\arg\min}_{\V\in\mS_d, \, \rank(\V)=1}\V \bullet \nabla{}f(\X_t) + \frac{\eta_t\beta}{2}\Vert{\V - \X_t}\Vert_F^2.
\end{eqnarray}
Note that in this case it follows that $\V_t$ is a unit trace rank-one matrix which corresponds to the leading eigenvector of the matrix $-\nabla{}f(\X_t) + \eta_t\beta\X_t$. However, when $\V_t$ is simply rank-one, the regularization $\Vert{\V_t-\X_t}\Vert_F^2$ makes little sense in general, since unless $\X^*$ is rank-one, we do not expect $\X_t$ to be such. Note however, that if $\X^*$ is rank one, then this modification will already result in a linear convergence rate. However, we can think of solving a set of decoupled component-wise regularized problems:
\begin{eqnarray}\label{eq:towardsNewUpdate2}
\forall i\in[k]: \quad \v_t^{(i)} &\gets& {\arg\min}_{\Vert{\v}\Vert=1}\v^{\top}\nabla{}f(\X_t)\v + \frac{\eta_t\beta}{2}\Vert{\v\v^{\top} - \x_i\x_i^{\top}}\Vert_F^2 \nonumber \\
&& \equiv \EV\left({-\nabla{}f(\X_t) + \eta_t\beta\x_i\x_i^{\top}}\right) \nonumber \\
\X_{t+1} &\gets &  \sum_{i=1}^ka_i\left({(1-\eta_t)\x_i\x_i + \eta_t\v_t^{(i)}\v_t^{(i)\top}}\right).
\end{eqnarray}

Following the lines of Eq. \eqref{eq:oldCGanalysis}, we will now have that

\begin{eqnarray}\label{eq:distributedCGanalysis}
h_{t+1} &\leq &h_t + \eta_t\sum_{i=1}^ka_i(\v_t^{(i)}\v_t^{(i)\top} - \x_i\x_i^{\top})\bullet \nabla{}f(\X_t) + \frac{\eta_t^2\beta}{2}\Vert{\sum_{i=1}^ka_i(\v_t^{(i)}\v_t^{(i)\top}-\x_i\x_i^{\top})}\Vert_F^2 \nonumber \\
&\leq &h_t + \eta_t\sum_{i=1}^ka_i(\v_t^{(i)}\v_t^{(i)\top} - \x_i\x_i^{\top})\bullet \nabla{}f(\X_t) + \frac{\eta_t^2\beta}{2}\sum_{i=1}^ka_i\Vert{\v_t^{(i)}\v_t^{(i)\top}-\x_i\x_i^{\top}}\Vert_F^2 \nonumber \\
& = &  h_t + \eta_t\E_{i\sim(a_1,...,a_k)}\left[{(\v_t^{(i)}\v_t^{(i)\top}- \x_i\x_i^{\top})\bullet \nabla{}f(\X_t) + \frac{\eta_t\beta}{2}\Vert{\v_t^{(i)}\v_t^{(i)\top}-\x_i\x_i^{\top}}\Vert_F^2}\right],
\end{eqnarray}
where the second inequality follows from convexity of the squared Frobenius norm, and the last equality follows since $(a_1,...,a_k)$ is a probability distribution over $[k]$.

While the approach in Eq. \eqref{eq:towardsNewUpdate2} relies only on leading eigenvector computations, the benefit in terms of potential convergence rates is not trivial, except for the case in which $\rank(\X^*) =1$ (then, by previous arguments, it is equivalent to computing the projection), since it is not immediate that we can get non-trivial bounds for the individual distances $\Vert{\v_t^{(i)}\v_t^{(i)\top}-\x_i\x_i^{\top}}\Vert_F$. Indeed, the main novelty in our analysis is dedicated precisely to this issue. 

A motivation, if any, is that there might exists a decomposition of $\X^*$ as $\X^* = \sum_{i=1}^kb_i\x^{*(i)}\x^{*(i)\top}$, which is close in some sense to the decomposition of $\X_t$. We can then think of the regularized problem in Eq. \eqref{eq:towardsNewUpdate2}, as an attempt to push each individual component $\x^{(i)}$ towards its corresponding component in the decomposition of $\X^*$, and as an overall result, bring the following iterate $\X_{t+1}$ closer to $\X^*$.

Note that Eq. \eqref{eq:distributedCGanalysis} implicitly describes a randomized algorithm in which, instead of solving a regularized EV problem for each rank-one matrix in the decomposition of $\X_t$, which is expensive as this decomposition grows large with the number of iterations, we pick a single rank-one component according to its weight in the decomposition, and only update it. This directly brings us to our proposed algorithm, Algorithm \ref{alg:1}, which is given below.

\begin{algorithm}
\caption{Randomized Rank one-regularized Conditional Gradient}
\label{alg:1}
\begin{algorithmic}[1]
\STATE input: sequence of step-sizes $\{\eta_t\}_{t\geq 1}$, sequence of error tolerances $\{\xi_t\}_{t\geq 0}$
\STATE let $\x_0$ be an arbitrary unit vector
\STATE $\X_1 \gets \x_1\x_1^{\top}$ such that $\x_1\gets\EV_{\xi_0}(-\nabla{}f(\x_0\x_0^{\top}))$
\FOR{$t = 1...$}
\STATE suppose $\X_t$ is given by $\X_t = \sum_{i=1}^ka_i\x_i\x_i^{\top}$, where each $\x_i$ is a unit vector, and $(a_1,a_2,...,a_k)$ is a probability distribution over $[k]$, for some integer $k$. 
\STATE pick $i_t\in[k]$ according to the probability distribution $(a_1,a_2,...a_k)$
\STATE set a new step-size $\tilde{\eta}_t$ as follows:
\begin{eqnarray*}
\tilde{\eta}_t \gets \left\{ \begin{array}{ll}
         \eta_t/2 & \mbox{if $a_{i_t} \geq \eta_t$}\\
        a_{i_t} & else \end{array} \right.
\end{eqnarray*}
\STATE $\v_t \gets \EV_{\xi_t}\left({-\nabla{}f(\X_t) + \eta_t\beta\x_{i_t}\x_{i_t}^{\top}}\right)$
\STATE $\X_{t+1} \gets \X_t + \tilde{\eta}_t(\v_t\v_t^{\top}-\x_{i_t}\x_{i_t}^{\top})$
\ENDFOR
\end{algorithmic}
\end{algorithm}

We have the following guarantee for Algorithm \ref{alg:1} which is the main result of this paper.
\begin{theorem}\label{thm:main}[Main Theorem]
Consider the sequence of step-sizes $\{\eta_t\}_{t\geq 1}$ defined by $\eta_t = 18/(t+8)$, and suppose that $\xi_0 = \beta$ and for any iteration $t\geq 1$ it holds that 
\begin{eqnarray*}
\xi_t = O\left({\min\{\frac{\beta{}}{t} ,\left({\frac{\beta\sqrt{\rank(\X^*)}}{\alpha^{1/4}t}}\right)^{4/3}, \left({\frac{\beta}{\sqrt{\alpha}\lambda_{\min}(\X^*)t}}\right)^{2}\}}\right) .
\end{eqnarray*} 
Then, all iterates of Algorithm \ref{alg:1} are feasible, and
\begin{eqnarray*}
\forall t\geq 1: \quad \E\left[{f(\X_t) - f(\X^*)}\right] = O\left({\min\{\frac{\beta{}}{t} ,\left({\frac{\beta\sqrt{\rank(\X^*)}}{\alpha^{1/4}t}}\right)^{4/3}, \left({\frac{\beta}{\sqrt{\alpha}\lambda_{\min}(\X^*)t}}\right)^{2}\}}\right).
\end{eqnarray*}
\end{theorem}

We now make several remarks regarding Algorithm \ref{alg:1} and Theorem \ref{thm:main}:
\begin{itemize}
\item
Observe that the feasibility of the iterates follows directly from the definition of $\tilde{\eta}_t$, since it is never allowed to exceed the corresponding coefficient $a_{i_t}$.
\item
None of the three bounds in Theorem \ref{thm:main} is better than the others for every value of $t$. In particular note that $\lambda_{\min}(\X^*)^{-1} \geq \rank(\X^*)$. Also, we note that the first $O(1/t)$ bound comes from the standard CG analysis.
\item
The dependency of the improved rates in Theorem \ref{thm:main} on $\rank(\X^*)$ and $\lambda_{\min}(\X^*)$ is not surprising since, in general, the standard $1/t$ rate of the CG method could not be improved without such additional dependencies. See for instance Section 7.4 in \cite{JaggiThesis}.
\item
The step-size choice in Theorem \ref{thm:main} does not require any knowledge on the parameters $\alpha,\beta, \rank(\X^*)$, and $\lambda_{\min}(\X^*)$. The knowledge of the smoothness parameter $\beta$ is required however for the computation of $\v_t$ on each iteration. While it follows from Theorem \ref{thm:main} that the knowledge of $\alpha,\rank(\X^*),\lambda_{\min}(\X^*)$ is needed to set the accuracy for the EV solver - $\xi_t$, in practice, iterative methods for eigenvector computation are very efficient and are much less sensitive to exact knowledge of parameters than the choice of step-size for instance. 
\item
While the eigenvalue problem solved on each iteration in Algorithm \ref{alg:1} is different from the one in the original CG algorithm (Algorithm \ref{alg:cg}), because of the additional term that depends on $\x_{i_t}\x_{i_t}^{\top}$, the efficiency of solving both EV problems is essentially the same. This follows since efficient EV procedures are based on iteratively multiplying  the desired input matrix $\M$ with some vector $\v$. In particular, multypling $\v$ with a rank-one matrix takes $O(d)$ time. Thus, as long as $\nnz(\nabla{}f(\X_t)) = \Omega(d)$, which is highly reasonable, it follows that both EV computations run in essentially the same time.

\item Aside from the computation of the gradient direction and the leading eigenvector computation, all other operations on any iteration $t$, can be carried out in $O(d^2 + t)$ additional time.

\item
Algorithm \ref{alg:1} does not directly use the input step-size sequence, but uses instead a modified sequence $\{\tilde{\eta}_t\}_{t\geq 1}$. This modification is made for clarity of the analysis. One can think of the sequence $\{\eta_t\}_{t\geq 1}$ as the sequence that we would like to use, however, we need to modify it a bit so on one hand, we can make sure that the iterates of the algorithm are indeed feasible at all times, and on the other hand, we can make sure that the algorithm can make sufficient progress on each iteration. 
\end{itemize}

\section{Analysis}\label{sec:analysis}

Throughout this section, given a matrix $\Y\in\mS_d$, we let $\P_{\Y,\tau}\in\mbS_d$ denote the projection matrix onto all eigenvectors of $\Y$ that correspond to eigenvalues of magnitude at least $\tau$. Similarly, we let $\P_{\Y,\tau}^{\perp}$ denote the projection matrix onto the eigenvectors of $\Y$ that correspond to eigenvalues of magnitude smaller that $\tau$ (including eigenvectors that correspond to zero-valued eigenvalues).

\subsection{A new decomposition for positive semidefinite matrices with locality proprieties}

The analysis of Algorithm \ref{alg:1} relies heavily on a new decomposition idea of matrices in $\mS_d$ that suggests that given a matrix $\X$ in the form of a convex combination of rank-one matrices: $\X = \sum_{i=1}^k\alpha_i\x_i\x_i^{\top}$, and another matrix $\Y\in\mS_d$, roughly speaking, we can decompose $\Y$ as the sum of rank-one matrices, such that the components in the decomposition of $\Y$ are close to those in the decomposition of $\X$ in terms of the overall distance $\Vert{\X-\Y}\Vert_F$. This decomposition and corresponding property justifies the idea of solving rank-one regularized problems, as suggested in Eq. \eqref{eq:towardsNewUpdate2}, and applied in Algorithm \ref{alg:1}.

In order to present this decomposition and its nice local proprieties, we first need two technical lemmas, and then we present the main lemma of this subsection, Lemma \ref{lem:optStrucutre2}, which gives the exact bounds that will be used in the convergence analysis of Algorithm \ref{alg:1}.

\begin{lemma}\label{lem:optStructure}
Let $\X,\Y\in\mS_d$. Let $\tau, \gamma\in[0,1]$ be scalars that satisfy $\frac{\gamma\tau}{1-\gamma}\geq \Vert{\X-\Y}\Vert_F$. 
Then it holds that $\Y \succeq (1-\gamma)\P_{\Y,\tau}\X\P_{\Y,\tau}$.
\end{lemma}

\begin{proof}
Given a vector $\w\in\reals^d$ let us write it as $\w = \w^+ + \w^-$ where $\w^+ = \P_{\Y,\tau}\w$ and $\w^- = \P_{\Y,\tau}^{\perp}\w = \w - \w^+$.

It holds that
\begin{eqnarray}\label{eq:lem:os:1}
\w^{\top}\Y\w &=& \w^{+\top}\Y\w^+ + \w^{-\top}\Y\w^- + 2\w^{-\top}\Y\w^+ \nonumber \\
&=& \w^{+\top}\Y\w^+ + \w^{-\top}\Y\w^- + 2\w^{\top}\P_{\Y,\tau}^{\perp}\Y\P_{\Y,\tau}\w \nonumber \\
&\geq & \w^{+\top}\Y\w^+ ,
\end{eqnarray}
where the inequality follows since $\P_{\Y,\tau}^{\perp}\Y\P_{\Y,\tau} = 0$ and $\Y$ is positive semidefinite.

Similarly, since $\P_{\Y,\tau}\w^- = 0$, we have that
\begin{eqnarray}\label{eq:lem:os:2}
\w^{-\top}\P_{\Y,\tau}\X\P_{\Y,\tau}\w^- = \w^{-\top}\P_{\Y,\tau}\X\P_{\Y,\tau}\w^+ = \w^{+\top}\P_{\Y,\tau}\X\P_{\Y,\tau}\w^- =0 .
\end{eqnarray}

Note also that 
\begin{eqnarray}\label{eq:lem:os:3}
\w^{+\top}\P_{\Y,\tau}\X\P_{\Y,\tau}\w^+ &=& \w^{+\top}\X\w^+ .
\end{eqnarray}

Thus, we have that
\begin{eqnarray*}
\w^{\top}\left[{(1-\gamma)\P_{\Y,\tau}\X\P_{\Y,\tau}}\right]\w &=& (1-\gamma)\w^{+\top}\P_{Y,\tau}\X\P_{\Y,\tau}w^+ \\
&=& (1-\gamma)\w^{+\top}\X\w^+ \\
&\leq & (1-\gamma)\left({\w^{+\top}\Y\w^+ + \Vert{\X-\Y}\Vert_F \cdot \Vert{\w^+}\Vert^2}\right) \\
& \leq & \w^{\top}\Y\w  + (1-\gamma)\Vert{\X-\Y}\Vert_F \cdot \Vert{\w^+}\Vert^2 - \gamma\w^{+\top}\Y\w^+ \\
&\leq &\w^{\top}\Y\w  + (1-\gamma)\Vert{\X-\Y}\Vert_F\cdot \Vert{\w^+}\Vert^2 - \gamma\tau\Vert{\w^+}\Vert^2,
\end{eqnarray*}
where the first equality follows from Eq. \eqref{eq:lem:os:2}, the second equality follows from Eq. \eqref{eq:lem:os:3}, the first inequality follows from the Cauchy-Schwarz ineq., the second inequality follows from Eq. \eqref{eq:lem:os:1}, and the last inequality follows from the definitions of $\w^+$ and $\tau$.

Thus, we can see that if $\frac{\gamma\tau}{1-\gamma} \geq \Vert{\X-\Y}\Vert_F$, the lemma follows.

\end{proof}

\begin{lemma}\label{lem:closePairsDecompose}
Let $\X,\Y\in\mS_d$ and suppose $\X$ is given in the form $\X = \sum_{i=1}^ka_i\x_i\x_i^{\top}$ where each $\x_i$ is a unit vector, and the weights $(a_1,...,a_k)$ are a distribution over $[k]$. Let $\P\in\mbS_d$ be a projection matrix onto a subset of the eigenvectors of $\Y$, and define for any $i\in[k]$, $\tilde{\x}_i := \P\x_i$. Then, it holds that
\begin{eqnarray*}
\sum_{i=1}^ka_i(1-\Vert{\tilde{\x}_i}\Vert^2) \leq \sqrt{\rank(\Y)}\Vert{\Y-\P\X\P}\Vert_F .
\end{eqnarray*}

\end{lemma}

\begin{proof}
Let us write the eigen-decomposition of $\Y$ as $\Y = \sum_{j=1}^{\rank(\Y)}\lambda_j\v_j\v_j^{\top}$. Using simple algebraic manipulations we have that

\begin{eqnarray*}
\Vert{\Y-\P\X\P}\Vert_F^2 &\geq & \sum_{j=1}^{\rank(\Y)}\left({(\Y-\P\X\P)\bullet \v_j\v_j^{\top}}\right)^2 
= \sum_{j=1}^{\rank(\Y)}\left({\lambda_j - \sum_{i=1}^ka_i\v_j^{\top}\P\x_i\x_i^{\top}\P\v_j}\right)^2 \\
&= & \sum_{j=1}^{\rank(\Y)}\left({\lambda_j - \sum_{i=1}^ka_i(\v_j^{\top}\P\x_i)^2}\right)^2 \\
&\geq & \frac{1}{\rank(\Y)}\left({\sum_{j=1}^{\rank(\Y)}\left({\lambda_j - \sum_{i=1}^ka_i(\v_j^{\top}\P\x_i)^2}\right)}\right)^2 \\
&=& \frac{1}{\rank(\Y)}\left({1 - \sum_{j=1}^{\rank(\Y)}\sum_{i=1}^ka_i(\v_j^{\top}\P\x_i)^2}\right)^2 \\
& =  & \frac{1}{\rank(\Y)}\left({\sum_{i=1}^ka_i\left({1 - \sum_{j=1}^{\rank(\Y)}(\v_j^{\top}\P\x_i)^2}\right)}\right)^2 \\
& =  & \frac{1}{\rank(\Y)}\left({\sum_{i=1}^ka_i\left({1 - \Vert{\tilde{\x}_i}\Vert^2}\right)}\right)^2 .
\end{eqnarray*}

Thus we have that
\begin{eqnarray*}
\sum_{i=1}^ka_i(1-\Vert{\tilde{\x}_i}\Vert^2) \leq \sqrt{\rank(\Y)}\Vert{\Y-\P\X\P}\Vert_F ,
\end{eqnarray*}
which gives the bound in the lemma.

\end{proof}

\begin{lemma}\label{lem:optStrucutre2}
Let $\X,\Y \in\mS_d$ such that $\X$ is given as $\X= \sum_{i=1}^ka_i\x_i\x_i^{\top}$, where each $\x_i$ is a unit vector, and $(a_1,...,a_k)$ is a distribution over $[k]$, and let $\tau,\gamma\in[0,1]$ which satisfy the condition in Lemma \ref{lem:optStructure}. Then, $\Y$ can be written as $$\Y = \sum_{i=1}^kb_i\y_i\y_i^{\top} + \sum_{j=1}^k(a_j-b_j)\W$$ such that 
\begin{enumerate}
\item
each $\y_i$ is a unit vector, $(b_1,...,b_k)$ is a distribution over $[k]$, and $\W \in \mS_d$
\item
$\forall i\in[k]: b_i \leq a_i$ and $\sum_{j=1}^k(a_j-b_j) \leq \sqrt{\rank(\Y)}\left({\Vert{\Y\P_{\Y,\tau}^{\perp}}\Vert_F + \Vert{\X-\Y}\Vert_F}\right) + \gamma$
\item
$\sum_{i=1}^kb_i\Vert{\x_i\x_i^{\top} - \y_i\y_i^{\top}}\Vert_F^2 \leq 2\sqrt{\rank(\Y)}\left({\Vert{\Y\P_{\Y,\tau}^{\perp}}\Vert_F + \Vert{\X-\Y}\Vert_F}\right)$
\end{enumerate}
\end{lemma}

\begin{proof}
For each $i\in[k]$ let $\tilde{\x}_i = \P_{\Y,\tau}\x_i$. It follows from Lemma 
\ref{lem:optStructure} that as long as $\frac{\gamma\tau}{1-\gamma} \geq \Vert{\X-\Y}\Vert_F$, it holds that
\begin{eqnarray*}
\Y \succeq \sum_{i=1}^ka_i(1-\gamma)\tilde{\x}_i\tilde{\x}_i^{\top} .
\end{eqnarray*}

Since $\Y\in\mS_d$ and $\trace\left({\sum_{i=1}^ka_i(1-\gamma)\tilde{\x}_i\tilde{\x}_i^{\top}}\right) = \sum_{i=1}^ka_i(1-\gamma)\Vert{\tilde{\x}_i}\Vert^2$, it follows that $\Y$ can be written as:
\begin{eqnarray*}
\Y = \sum_{i=1}^ka_i(1-\gamma)\tilde{\x}_i\tilde{\x}_i^{\top} + \left({\sum_{j=1}^ka_j\left({1-(1-\gamma)\Vert{\tilde{\x}_j}\Vert^2}\right)}\right)\W,
\end{eqnarray*}
where $\W\in\mS_d$. 

Let us now define $\y_i := \frac{\tilde{\x}_i}{\Vert{\tilde{\x}_i}\Vert}$ and $b_i := a_i(1-\gamma)\Vert{\tilde{\x}_i}\Vert^2$. Then indeed it follows that
\begin{eqnarray*}
\Y = \sum_{i=1}^kb_i\y_i\y_i^{\top} + \sum_{j=1}^k(a_j-b_j)\W .
\end{eqnarray*}

We are going to apply Lemma \ref{lem:closePairsDecompose} to derive the bounds listed in the lemma. As a first step, we need to bound the distance $\Vert{\Y-\P_{\Y,\tau}\X\P_{\Y,\tau}}\Vert_F$.

\begin{eqnarray}\label{eq:1}
\Vert{\Y-\P_{\Y,\tau}\X\P_{\Y,\tau}}\Vert_F &\leq & \Vert{\Y-\P_{\Y,\tau}\Y\P_{\Y,\tau}}\Vert_F + \Vert{\P_{\Y,\tau}\X\P_{\Y,\tau}-\P_{\Y,\tau}\Y\P_{\Y,\tau}}\Vert_F \nonumber\\
&\leq & \Vert{\Y\P_{\Y,\tau}^{\perp}}\Vert_F + \Vert{\X-\Y}\Vert_F,
\end{eqnarray}

where the bound on $ \Vert{\Y-\P_{\Y,\tau}\Y\P_{\Y,\tau}}\Vert_F$ follows from the definition of $\P_{\Y,\tau}$, and the bound on $\Vert{\P_{\Y,\tau}\X\P_{\Y,\tau}-\P_{\Y,\tau}\Y\P_{\Y,\tau}}\Vert_F$ follows from the inequality $\Vert{\A\B}\Vert_F \leq \Vert{\A}\Vert\cdot\Vert{\B}\Vert_F$.

By definition of $\{b_i\}_{i\in[k]}$ it holds that
\begin{eqnarray*}
\sum_{i=1}^k(a_i - b_i) &=& \sum_{i=1}^ka_i(1 - (1-\gamma)\Vert{\tilde{\x}_i}\Vert^2) \leq \sum_{i=1}^ka_i(1-\Vert{\tilde{\x}_i}\Vert^2) + \gamma \\
&\leq & \sqrt{\rank(\Y)}\left({\Vert{\Y\P_{\Y,\tau}^{\perp}}\Vert_F + \Vert{\X-\Y}\Vert_F}\right) + \gamma,
\end{eqnarray*}
where the last inequality follows from Lemma \ref{lem:closePairsDecompose} and the bound in Eq. \eqref{eq:1}.

We continue to upper-bound $\sum_{i=1}^kb_i\Vert{\x_i\x_i^{\top} - \y_i\y_i^{\top}}\Vert_F^2$:
\begin{eqnarray*}
\sum_{i=1}^kb_i\Vert{\x_i\x_i^{\top} - \y_i\y_i^{\top}}\Vert_F^2 &\leq &\sum_{i=1}^ka_i\Vert{\x_i\x_i^{\top} - \y_i\y_i^{\top}}\Vert_F^2 \\
&=& 2\sum_{i=1}^ka_i(1-(\x_i^{\top}\y_i)^2) \\
&=& 2\sum_{i=1}^ka_i\left({1-\left({\frac{\x_i^{\top}\tilde{\x}_i}{\Vert{\tilde{\x}_i}\Vert}}\right)^2}\right) \\
&=& 2\sum_{i=1}^ka_i(1-\Vert{\tilde{\x}_i}\Vert^2) \\
&\leq & 2\sqrt{\rank(\Y)}\left({\Vert{\Y\P_{\Y,\tau}^{\perp}}\Vert_F + \Vert{\X-\Y}\Vert_F}\right),
\end{eqnarray*}
where the last inequality follows again from the application of Lemma \ref{lem:closePairsDecompose} and the bound in Eq. \eqref{eq:1}.
\end{proof}

\subsection{Bounding the per-iteration improvement}

We now turn to analyze the per-iteration improvement of Algorithm \ref{alg:1}. We start by first analyzing a deterministic, and much less efficient, version that updates all of the rank-one components on each iteration $t$, as suggested in Eq. \eqref{eq:towardsNewUpdate2}. This is done in Lemma \ref{lem:fullUpdateBound}. Then in Lemma \ref{lem:randomUpdateBound}, we apply Lemma \ref{lem:fullUpdateBound}, to analyze the randomized step of Algorithm \ref{alg:1}. However, first we need a simple observation regarding Algorithm \ref{alg:1}, that shows that it can always take sufficiently large step-sizes, i.e., step-size of magnitude at least $\eta_t/2$ on iteration $t$.

\begin{observation}\label{observ:largeCoeff}
In case the input sequence of step-sizes in Algorithm \ref{alg:1} - $\{\eta_t\}_{t\geq 1}$, is monotonically non-increasing and $\eta_t \in[0,2]$ for all $t\geq 1$, it follows that on each iteration $t$ of the algorithm, the iterate $\X_t$ admits an explicitly-given factorization into a convex sum of rank-one matrices, as described in the algorithm, such that for every rank-one coefficient $a_i$, it holds that $a_i \geq \eta_t/2$.
\end{observation}

\begin{proof}
The proof is by a simple induction. Since $\X_1$ is just a rank-one matrix, it follows that the corresponding coefficient in the convex sum is $a_1=1$. Thus, for any $\eta_1\in[0,2]$ it indeed follows that $a_1 \geq \eta_1/2$. 
Assume now that the induction holds for time $t\geq 1$. On time $t$ we choose a coefficient $a_{i_t}$ and move a mass of $\tilde{\eta}_t$ from it to a new rank-one matrix $\y_t\y_t^{\top}$, and all other coefficients remain unchanged. Since we assume that the step-size sequence is monotonically non-increasing, it directly follows that the induction step holds for all unchanged coefficients. Regarding the affected coefficients $a_{i_t}$ and the coefficient of the new rank-one matrix, we consider two cases. First, if $a_{i_t} \geq \eta_t$ then by the definition of $\tilde{\eta}_t$ we have that the mass of the new coefficient is going to be exactly $\eta_t/2$ and the mass of the old coefficient is going to be  $a_{i_t}-\eta_t/2 \geq \eta_t/2$, and thus the induction holds. In the other case, we have that $\tilde{\eta}_t = a_{i_t} < \eta_t$. By the induction hypothesis we know that $a_{i_t} \geq \eta_t/2 \geq \eta_{t+1}/2$. Since we are moving now all the mass from $a_{i_t}$ to the new rank-one matrix, it follows that its weight is also going to be at least $\eta_{t+1}/2$, and thus the induction follows.
\end{proof}

\begin{lemma}\label{lem:fullUpdateBound}[full deterministic update]
Fix a scalar $\eta > 0$. Let $\X\in\mS_d$ such that $\X = \sum_{i=1}^ka_i\x_i\x_i^{\top}$, where each $\x_i$ is a unit vector, and $(a_1,...,a_k)$ is a probability distribution over $[k]$. For any $i\in[k]$, let 
\begin{eqnarray}\label{eq:2}
\v_i := \EV_{\xi}\left({-\nabla{}f(\X) + \eta\beta\x_i\x_i^{\top}}\right),
\end{eqnarray}
for some parameter $\xi > 0$.
Then, it holds that
\begin{eqnarray*}
&&\sum_{i=1}^k  a_i\left[{(\v_i\v_i^{\top}-\x_i\x_i^{\top})\bullet \nabla{}f(\X) + \frac{\eta\beta}{2}\Vert{\v_i\v_i^{\top}-\x_i\x_i^{\top}}\Vert_F^2}\right] \leq 
 -\left({f(\X)-f(\X^*)}\right) \\
 && + \eta\beta\cdot\min\{1, \, 5\sqrt{\sqrt{\frac{2}{\alpha}}\rank(\X^*)\sqrt{f(\X)-f(\X^*)}}, \, \frac{3\sqrt{2}}{\sqrt{\alpha}\lambda_{\min}(\X^*)}\sqrt{f(\X)-f(\X^*)} \} + \xi.
\end{eqnarray*}
\end{lemma}

\begin{proof}
For the sake of clarity, throughout the proof we treat each $\v_i$ as the result of an exact eigenvector computation, i.e., we assume $\xi = 0$, and at the end we discuss the effect of the approximation error in the computation of $\v_i$.

Let $\w^* \in\arg\min_{\w: \Vert{\w}\Vert=1}\w^{\top}\nabla{}f(\X)\w$. Using the optiamlity of $\v_i$, and the fact that for every $i\in[k]$,  both $\v_i$ and $\x_i$ are unit vectors, we have that


\begin{eqnarray}\label{eq:6}
&&\sum_{i=1}^ka_i\left[{(\v_i\v_i^{\top}-\x_i\x_i^{\top})\bullet \nabla{}f(\X) + \frac{\eta\beta}{2}\Vert{\v_i\v_i^{\top}-\x_i\x_i^{\top}}\Vert_F^2}\right] \leq \nonumber \\
&&\sum_{i=1}^ka_i\left[{(\w^*\w^{*\top}-\x_i\x_i^{\top})\bullet \nabla{}f(\X) + \frac{\eta\beta}{2}\Vert{\w^*\w^{*\top}-\x_i\x_i^{\top}}\Vert_F^2}\right] \leq \nonumber \\
&&\sum_{i=1}^ka_i\left[{(\w^*\w^{*\top}-\x_i\x_i^{\top})\bullet \nabla{}f(\X) + \eta\beta}\right] = \nonumber \\
&&(\w^*\w^{*\top}-\X)\bullet\nabla{}f(\X) + \eta\beta \leq 
 (\X^* - \X)\bullet\nabla{}f(\X) + \eta\beta \leq -\left({f(\X)-f(\X^*)}\right) + \eta\beta,  \nonumber \\
\end{eqnarray}
where the third inequality follows from the optimality of $\w^*$, and the last inequality follows from the convexity of $f$. Thus, Eq. \eqref{eq:6} gives us the first part of the bound stated in the lemma. We now move on the prove the second part.

From Lemma \ref{lem:optStrucutre2} we know we can write $\X^*$ in the following way:
\begin{eqnarray}\label{eq:5}
\X^* = \sum_{i=1}^kb_i^*\y_i^*\y_i^{*\top} + \sum_{j=1}^k(a_j^*-b_j^*)\W^*,
\end{eqnarray}
where for all $i\in[k]$, $b_i^*\in[0,a_i]$ and $\y_i^*$ is a unit vector, and $\W^*\in\mS_d$.

Using again the optimality of $\v_i$ for each $i\in[k]$, we have that
\begin{eqnarray}\label{eq:4}
&&\sum_{i=1}^ka_i\left[{(\v_i\v_i^{\top}-\x_i\x_i^{\top})\bullet \nabla{}f(\X) + \frac{\eta\beta}{2}\Vert{\v_i\v_i^{\top}-\x_i\x_i^{\top}}\Vert_F^2}\right] \leq \nonumber \\
&&\sum_{i=1}^ka_i\cdot\min\{(\y_i^*\y_i^{*\top}-\x_i\x_i^{\top})\bullet \nabla{}f(\X) + \frac{\eta\beta}{2}\Vert{\y_i^*\y_i^{*\top}-\x_i\x_i^{\top}}\Vert_F^2 , \nonumber \\
&&(\w^*\w^{*\top}-\x_i\x_i^{\top})\bullet \nabla{}f(\X) + \frac{\eta\beta}{2}\Vert{\w^*\w^{*\top}-\x_i\x_i^{\top}}\Vert_F^2\}\leq \nonumber \\
&&\sum_{i=1}^kb_i^*\left[{(\y_i^*\y_i^{*\top}-\x_i\x_i^{\top})\bullet \nabla{}f(\X) +\frac{\eta\beta}{2}\Vert{\y_i^*\y_i^{*\top}-\x_i\x_i^{\top}}\Vert_F^2}\right] \nonumber \\
&& + \sum_{i=1}^k(a_i-b_i^*)\left[{(\w^*\w^{*\top}-\x_i\x_i^{\top})\bullet \nabla{}f(\X) +\frac{\eta\beta}{2}\Vert{\w^*\w^{*\top}-\x_i\x_i^{\top}}\Vert_F^2}\right] \nonumber \\
&& \leq \sum_{i=1}^kb_i^*\left[{(\y_i^*\y_i^{*\top}-\x_i\x_i^{\top})\bullet \nabla{}f(\X) +\frac{\eta\beta}{2}\Vert{\y_i^*\y_i^{*\top}-\x_i\x_i^{\top}}\Vert_F^2}\right] \nonumber \\
&&+\sum_{i=1}^k(a_i-b_i^*)\left[{(\W^*-\x_i\x_i^{\top})\bullet \nabla{}f(\X) +\frac{\eta\beta}{2}\Vert{\w^*\w^{*\top}-\x_i\x_i^{\top}}\Vert_F^2}\right], 
\end{eqnarray}

where the second inequality follows since $\min\{a,b\} \leq \lambda{}a+(1-\lambda)b$ for any $a,b\in\reals,\lambda\in[0,1]$, and the third inequality follows
from the optimality of $\w^*$. Using Eq. \eqref{eq:5} we have that
\begin{eqnarray}\label{eq:7}
\textrm{RHS of } \eqref{eq:4} &\leq & (\X^*-\X)\bullet \nabla{}f(\X) + \sum_{i=1}^kb_i^*\frac{\eta\beta}{2}\Vert{\y_i^*\y_i^{*\top}-\x_i\x_i^{\top}}\Vert_F^2 \nonumber \\
&&+ \sum_{i=1}^k(a_i-b_i^*)\frac{\eta\beta}{2}\Vert{\w^*\w^{*\top}-\x_i\x_i^{\top}}\Vert_F^2 \nonumber \\
&\leq &  (\X^*-\X)\bullet \nabla{}f(\X) + \frac{\eta\beta}{2}\sum_{i=1}^kb_i^*\Vert{\y_i^*\y_i^{*\top}-\x_i\x_i^{\top}}\Vert_F^2
 + \eta\beta\sum_{i=1}^k(a_i-b_i^*) \nonumber \\
 & \leq & (\X^* - \X)\bullet \nabla{}f(\X) + \eta\beta\sqrt{\rank(\X^*)}\left({\Vert{\X^*\P_{\X^*,\tau}^{\perp}}\Vert_F + \Vert{\X-\X^*}\Vert_F}\right) \nonumber \\
 &&+ \eta\beta\left({\sqrt{\rank(\X^*)}\left({\Vert{\X^*\P_{\X^*,\tau}^{\perp}}\Vert_F + \Vert{\X-\X^*}\Vert_F}\right) + \gamma}\right) \nonumber  \\
& = & (\X^* - \X)\bullet \nabla{}f(\X) + \eta\beta\left({2\sqrt{\rank(\X^*)}\left({\Vert{\X^*\P_{\X^*,\tau}^{\perp}}\Vert_F + \Vert{\X-\X^*}\Vert_F}\right)+ \gamma}\right), \nonumber \\
\end{eqnarray}

where the last inequality follows from plugging the bounds in Lemma \ref{lem:optStrucutre2} and holds for any $\tau,\gamma\in[0,1]$ such that $\frac{\tau\gamma}{1-\gamma} \geq \Vert{\X-\X^*}\Vert_F$.

Now we can optimize the above bound in terms $\tau,\gamma$ under the constraint that $\frac{\gamma\tau}{1-\gamma} \geq \Vert{\X-\X^*}\Vert_F$. 
One option is to upper bound $\Vert{\X^*\P_{\X^*,\tau}^{\perp}}\Vert_F \leq \sqrt{\rank(\X^*)}\tau$, which gives us
\begin{eqnarray*}
\textrm{RHS of } \eqref{eq:4} &\leq & (\X^* - \X)\bullet \nabla{}f(\X) + \eta\beta\left({2\sqrt{\rank(\X^*)}\left({\sqrt{\rank(\Y)}\tau + \Vert{\X-\Y}\Vert_F}\right)+ \gamma}\right) .
\end{eqnarray*}

We can then set: 
\begin{eqnarray*}
\tau_1 = \sqrt{\frac{\Vert{\X-\X^*}\Vert_F}{2\rank(\X^*)}}, \qquad \gamma_1 =\sqrt{2\rank(\X^*)\Vert{\X-\X^*}\Vert_F}, 
\end{eqnarray*}
as long as $\Vert{\X-\X^*}\Vert_F \leq \frac{1}{2\rank(\X^*)}$,
which gives us:
\begin{eqnarray*}
\textrm{RHS of } \eqref{eq:4} &\leq & (\X^* - \X)\bullet \nabla{}f(\X) + 2\eta\beta\sqrt{\rank(\X^*)}\left({\sqrt{2\Vert{\X-\X^*}\Vert_F} +  \Vert{\X-\X^*}\Vert_F}\right). 
\end{eqnarray*}
Note that in order for the above bound to improve over that in Eq. \eqref{eq:6}, it indeed must in particular hold that $\Vert{\X-\X^*}\Vert_F <  \frac{1}{2\rank(\X^*)}$. In that case it follows that
\begin{eqnarray}\label{eq:8}
&&\textrm{RHS of } \eqref{eq:4} \leq (\X^* - \X)\bullet \nabla{}f(\X) + 5\eta\beta\sqrt{\rank(\X^*)\Vert{\X-\X^*}\Vert_F}.
\end{eqnarray}

Another option, is to choose
\begin{eqnarray*}
\tau_2 = \lambda_{\min}(\X^*), \qquad \gamma_2 = \frac{\Vert{\X-\X^*}\Vert_F}{\lambda_{\min}(\X^*)}, 
\end{eqnarray*}
as long as $\Vert{\X-\X^*}\Vert_F < \lambda_{\min}(\X^*)$. In this case, it holds that $\Vert{\X^*\P_{\X^*,\tau}^{\perp}}\Vert_F = 0$. Plugging into Eq. \eqref{eq:7} we have that
\begin{eqnarray*}
\textrm{RHS of } \eqref{eq:4} &\leq & (\X^* - \X)\bullet \nabla{}f(\X) + \eta\beta\Vert{\X-\X^*}\Vert_F\left({2\sqrt{\rank(\X^*)} + \frac{1}{\lambda_{\min}(\X^*)}}\right).
\end{eqnarray*}
Note that since $\X^*\in\mS_d$ it holds thst $\lambda_{\min}(\X^*)^{-1} \geq \rank(\X^*)$ and thus we have that
\begin{eqnarray}\label{eq:9}
\textrm{RHS of } \eqref{eq:4} &\leq & (\X^* - \X)\bullet \nabla{}f(\X) + \frac{3\eta\beta\Vert{\X-\X^*}\Vert_F}{\lambda_{\min}(\X^*)}.
\end{eqnarray}
Note that here also, the above bound improves over the one in Eq. \eqref{eq:6} only when indeed $\Vert{\X-\X^*}\Vert_F < \lambda_{\min}(\X^*)$.

Now, by using the convexity of $f$ to upper bound $(\X^* - \X)\bullet \nabla{}f(\X) \leq -(f(\X)-f(\X^*))$ and Eq. \eqref{eq:strongconvexdist} to upper bound $\Vert{\X-\X^*}\Vert_F \leq \sqrt{\frac{2}{\alpha}(f(\X)-f(\X^*)}$ in both Eq. \eqref{eq:8} and \eqref{eq:9}, gives the rest of the bound in the lemma.

By going through the analysis above again (basically Eq. \eqref{eq:6} and Eq. \eqref{eq:4}), it's clear that an $\xi$ additive error in the computation of each eigenvector $\v_i$ results in a single additive term $\xi$ in all of the above bounds, and hence the lemma follows.
\end{proof}

\begin{lemma}\label{lem:randomUpdateBound}[randomized update]
Consider an iteration $t$ of Algorithm \ref{alg:1}. Fix a step-size $\eta_t$ and assume that the iterate of the algorithm on this iteration is feasible and given in the following explicit form: $\X_t = \sum_{i=1}^ka_i\x_i\x_i^{\top}$, where each $\x_i$ is a unit vector, and $(a_1,...,a_k)$ is a distribution over $[k]$. Further, suppose that each $a_i$ satisfies that $a_i \geq \eta_t/2$. Then,
\begin{eqnarray*}
\E[h_{t+1}] \leq \left({1 - \frac{\eta_t}{2}}\right)\E[h_t] +  \frac{\eta_t^2\beta}{2}\min\{1, \frac{5\sqrt{\sqrt{2}\rank(\X^*)}}{\alpha^{1/4}}\E[h_t]^{1/4}, \frac{3\sqrt{2}}{\sqrt{\alpha}\lambda_{\min}(\X^*)}\E[h_t]^{1/2}\} + \eta_t\xi_t,
\end{eqnarray*}
where $\forall t\geq 1$ $h_t := f(\X_t) - f(\X^*)$.
\end{lemma}

\begin{proof}

Using the update step of Algorithm \ref{alg:1} we have that
\begin{eqnarray*}
h_{t+1} &=& f(\X_{t+1}) - f(\X^*) = f(\X_t + \tilde{\eta}_t(\v_t\v_t^{\top} - \x_{i_t}\x_{i_t}))  - f(\X^*)\\
&\leq & f(\X_t) - f(\X^*) + \tilde{\eta}_t(\v_t\v_t^{\top} - \x_{i_t}\x_{i_t}) \bullet \nabla{}f(\X_t) + \frac{\tilde{\eta}_t^2\beta}{2}\Vert{\v_t\v_t^{\top} - \x_{i_t}\x_{i_t}}\Vert_F^2 \\
&\leq & h_t + \tilde{\eta}_t\left[{(\v_t\v_t^{\top} - \x_{i_t}\x_{i_t}) \bullet \nabla{}f(\X_t) + \frac{\eta_t\beta}{2}\Vert{\v_t\v_t^{\top} - \x_{i_t}\x_{i_t}}\Vert_F^2}\right], 
\end{eqnarray*}
where the first inequality follows from the smoothness of $f$ and the second one follows since by definition, $\eta_t \geq \tilde{\eta}_t$. 

By the choice of $\v_t$ we have that
\begin{eqnarray}\label{eq:lem:randUp:1}
(\v_t\v_t^{\top} - \x_{i_t}\x_{i_t}) \bullet \nabla{}f(\X_t) + \frac{\eta_t^2\beta}{2}\Vert{\v_t\v_t^{\top} - \x_{i_t}\x_{i_t}}\Vert_F^2 \leq \xi_t,
\end{eqnarray}
and thus, since by our assumption on $\{a_i\}_{i\in[k]}$, it also holds that $\tilde{\eta}_t \geq \eta_t/2$,  we have that
\begin{eqnarray}\label{eq:lem:randUp:2}
h_{t+1} &\leq & h_t + \frac{\eta_t}{2}\left[{(\v_t\v_t^{\top} - \x_{i_t}\x_{i_t}) \bullet \nabla{}f(\X_t) + \frac{\eta_t\beta}{2}\Vert{\v_t\v_t^{\top} - \x_{i_t}\x_{i_t}}\Vert_F^2}\right] + \left({\tilde{\eta}_t - \frac{\eta_t}{2}}\right)\xi_t \nonumber \\
&\leq & h_t + \frac{\eta_t}{2}\left[{(\v_t\v_t^{\top} - \x_{i_t}\x_{i_t}) \bullet \nabla{}f(\X_t) + \frac{\eta_t\beta}{2}\Vert{\v_t\v_t^{\top} - \x_{i_t}\x_{i_t}}\Vert_F^2}\right] + \frac{\eta_t}{2}\xi_t,
\end{eqnarray}
where the last inequality follows again by using $\eta_t \geq \tilde{\eta}_t$.

Taking expectation over the random choice of $i_t$ in Eq. \eqref{eq:lem:randUp:2}, and plugging Lemma \ref{lem:fullUpdateBound}, we have that
\begin{eqnarray*}
\E_{i_t}[h_{t+1} \, | \, \X_t] &\leq & h_t - \frac{\eta_t}{2}h_t +  \frac{\eta_t^2\beta}{2}\min\{1, \frac{5\sqrt{\sqrt{2}\rank(\X^*)}}{\alpha^{1/4}}h_t^{1/4}, \frac{3\sqrt{2}}{\sqrt{\alpha}\lambda_{\min}(\X^*)}h_t^{1/2}\} + \frac{\eta_t}{2}\xi_t +  \frac{\eta_t}{2}\xi_t.
\end{eqnarray*}


Taking expectation over the randomness introduced on iterations $1,...,{t-}1$ we have that
\begin{eqnarray*}
\E[h_{t+1}] &\leq & \left({1 - \frac{\eta_t}{2}}\right)\E[h_t] +  \frac{\eta_t^2\beta}{2}\min\{1, \frac{5\sqrt{\sqrt{2}\rank(\X^*)}}{\alpha^{1/4}}\E[h_t^{1/4}], \frac{3\sqrt{2}}{\sqrt{\alpha}\lambda_{\min}(\X^*)}\E[h_t^{1/2}]\} + \eta_t\xi_t \\
&\leq & \left({1 - \frac{\eta_t}{2}}\right)\E[h_t] +  \frac{\eta_t^2\beta}{2}\min\{1, \frac{5\sqrt{\sqrt{2}\rank(\X^*)}}{\alpha^{1/4}}\E[h_t]^{1/4}, \frac{3\sqrt{2}}{\sqrt{\alpha}\lambda_{\min}(\X^*)}\E[h_t]^{1/2}\} + \eta_t\xi_t ,
\end{eqnarray*}
where the first inequality follows since the function $f(x,y,z)=\min\{x,y,z\}$ is concave, and thus the inequality follows from applying Jensen's inequality. Similarly, the second inequality follows since both functions $g(x) = x^{1/4}$, $q(x) = x^{1/2}$ are also concave on $(0,\infty)$.
\end{proof}

\subsection{Proof of Theorem \ref{thm:main}}
We can now turn to prove our main theorem, Theorem \ref{thm:main}. The proof follows from deriving each one of the convergence rates in the theorem independently using the result of Lemma \ref{lem:randomUpdateBound}. This is done in the following Lemmas \ref{lem:recurs1},\ref{lem:recurs2}, \ref{lem:recurs3}. We then show that there exists a choice of step-size sequence and error-tolerance bounds for the eigenvector computations that satisfy all lemmas at once, and thus the theorem is obtained.

\begin{lemma}\label{lem:recurs1}
Let $C_,t_0$ be non-negative scalars that satisfy:
\begin{eqnarray*}
C \geq 18, \qquad  \frac{C}{2}-1 \geq t_0 \geq \frac{C}{6}-1 .
\end{eqnarray*}
Then if for all $t\geq 1$ we define $\eta_t = \frac{C}{3(t+t_0)}$, and we set $\xi_0 = \beta$ and $\forall t \geq 1: \, \xi_t = \frac{\beta{}C}{6(t+t_0)}$, it follows that all iterates of Algorithm \ref{alg:1} are feasible, and
\begin{eqnarray*}
\forall t\geq 1: \quad \E[h_t]\ \leq  \frac{\beta{}C}{t+t_0}.
\end{eqnarray*}
\end{lemma}
                                                                                                                                                                             
\begin{proof}
From Lemma \ref{lem:randomUpdateBound} we have that for all $t\geq 1$,
\begin{eqnarray*}
\forall t\geq 1: \quad \E[h_{t+1}] \leq \left({1 - \frac{\eta_t}{2}}\right)\E[h_t] +  \frac{\eta_t^2\beta}{2} +\eta_t\xi_t.
\end{eqnarray*}

We are going to assume throughout the proof that $\xi_t \leq \E[h_t]/6$. It thus follows that

\begin{eqnarray}\label{eq:recurs1:1}
\forall t\geq 1: \quad \E[h_{t+1}] \leq \left({1 - \frac{\eta_t}{3}}\right)\E[h_t] +  \frac{\eta_t^2\beta}{2}.
\end{eqnarray}

For all $t\geq 1$, define $v_t :=\beta^{-1}\E[h_t]$. Dividing both sides of Eq. \eqref{eq:recurs1:1} by $\beta$, we have that
\begin{eqnarray}\label{eq:recurs1:2}
\forall t\geq 1: \quad v_{t+1} \leq \left({1 - \frac{\eta_t}{3}}\right)v_t+  \frac{\eta_t^2}{2}.
\end{eqnarray}

We are going to prove by induction on $t$ that $v_t \leq \frac{C}{t+t_0}$ for suitable valus of $C, t_0$ and a sequence of step-sizes $\{\eta_t\}_{t\geq 1}$. Obviously for the base case $t=1$ to hold, we must restrict 
$\frac{C}{t_0+1} \geq v_1$.

Let us assume now that the induction hypothesis holds for some $t\geq 1$. 

Setting $\eta_t = \frac{C}{3(t+t_0)}$ in Eq. \eqref{eq:recurs1:2} we have that
\begin{eqnarray*}
v_{t+1} &\leq & v_t\left({1 - \frac{C}{9(t+t_0)}}\right) + \frac{C^2}{18(t+t_0)^2} \leq \frac{C}{t+t_0}\left({1 - \frac{C}{9(t+t_0)}}\right) + \frac{C^2}{18(t+t_0)^2} \\
& = & \frac{C}{t+t_0}\left({1 - \frac{C}{18(t+t_0)}}\right) = \frac{C}{t+t_0+1}\left({1 + \frac{1}{t+t_0}}\right)\left({1 - \frac{C}{18(t+t_0)}}\right) .
\end{eqnarray*}

Thus, choosing $C \geq 18$ gives:
\begin{eqnarray*}
v_{t+1} \leq  \frac{C}{t+1+t_0}\left({1+ \frac{1}{t+t_0}}\right)\left({1 - \frac{1}{t+t_0}}\right) < \frac{C}{t+1+t_0}
\end{eqnarray*}

as needed.

We can now set values for $C, t_0$ under the constraints that 
\begin{eqnarray}\label{eq:recurs1:3}
i.\, C \geq 18, \qquad ii. \, \frac{C}{t_0+1} \geq v_1, \qquad iii. \, \forall t\geq 1: \,\eta_t = \frac{C}{3(t+t_0)}\in[0,2] .
\end{eqnarray}

In order for our choice of step-sizes to satisfy the conditions of Observation \ref{observ:largeCoeff}, it must hold that $\{\eta_t\}_{t\geq 1} \subset[0,2]$. Since by definition this sequence is monotonic decreasing it suffices to show it for $\eta_1$. Thus we must require that $\frac{C}{3(1+t_0)} \leq 2$, which gives us the constraint $t_0 \geq \frac{C}{6}-1$.

It remains to deal with base case of the induction, i.e., we need to show that $v_1 = \beta^{-1}h_1 \leq \frac{C}{1+t_0}$ for our choice of $C,t_0$.

Recall that according to Algorithm \ref{alg:1} it holds that $\X_1 = \x_1\x_1^{\top}$, such that $\x_1=\EV(\nabla{}f(\x_0\x_0^{\top}))$, where $\x_0$ is some unit vector. Using the smoothness of $f$ we have that
\begin{eqnarray}\label{eq:recurs1:4}
h_1 &=& f(\x_1\x_1^{\top}) - f(\X^*) = f(\x_0\x_0^{\top} + \x_1\x_1^{\top} - \x_0\x_0^{\top}) -  f(\X^*) \nonumber \\
& \leq & f(\x_0\x_0^{\top}) - f(\X^*) + (\x_1\x_1^{\top} - \x_0\x_0^{\top}) \bullet \nabla{}f(\x_0\x_0^{\top}) +  \frac{\beta}{2}\Vert{\x_1\x_1^{\top}-\x_0\x_0^{\top}}\Vert_F^2 \nonumber\\
&\leq & f(\x_0\x_0^{\top}) - f(\X^*) + (\X^* - \x_0\x_0^{\top}) \bullet \nabla{}f(\x_0\x_0^{\top}) +  \frac{\beta}{2} \Vert{\x_1\x_1^{\top}-\x_0\x_0^{\top}}\Vert_F^2 + \xi_0 \nonumber\\
& \leq & \frac{\beta}{2} \Vert{\x_1\x_1^{\top}-\x_0\x_0^{\top}}\Vert_F^2 + \xi_0 \leq \beta + \xi_0,
\end{eqnarray}
where the second inequality follows from the choice of $\x_1$, and the third inequality follows from the convexity of $f(\X)$.

Setting $\xi_0 = \beta$, it follows that 
\begin{eqnarray*}
v_1 \leq \beta^{-1}\cdot 2\beta =  2.
\end{eqnarray*}
Thus we must require that $\frac{C}{1+t_0} \geq 2$, which gives us the constraint $t_0 \leq \frac{C}{2}-1$.

Thus, the conditions in Eq. \eqref{eq:recurs1:3} boils down to the following constraints:
\begin{eqnarray*}
C \geq 18, \qquad  \frac{C}{2}-1 \geq t_0 \geq \frac{C}{6}-1 .
\end{eqnarray*}

For $C,t_0$ that indeed satisfy these constraints we can thus conclude that
\begin{eqnarray*}
\forall t\geq 1: \quad \E[h_t] \leq  \beta{}v_t \leq  \frac{\beta{}C}{t+t_0} .
\end{eqnarray*}
\end{proof}

\begin{lemma}\label{lem:recurs2}
Let $C_,t_0$ be non-negative scalars that satisfy:
\begin{eqnarray*}
C \geq 30^{4/3}, \qquad  C^{3/4}-1 \geq t_0 \geq \frac{C^{3/4}}{6}-1 .
\end{eqnarray*}
Then if for all $t\geq 1$ we define $\eta_t = \frac{C^{3/4}}{3(t+t_0)}$, and set $\xi_0 = \beta$, $\forall t\geq 1: \, \xi_t = \frac{1}{6}\left({\frac{5C^{3/4}\beta\sqrt{\sqrt{2}\rank(\X^*)}}{\alpha^{1/4}(t+t_0)}}\right)^{4/3}$, it follows that all iterates of Algorithm \ref{alg:1} are feasible, and
\begin{eqnarray*}
\forall t\geq 1: \quad \E[h_t]\ \leq \left({\frac{5C^{3/4}\beta\sqrt{\sqrt{2}\rank(\X^*)}}{\alpha^{1/4}(t+t_0)}}\right)^{4/3}.
\end{eqnarray*}
\end{lemma}
                                                                                                                                                                             
\begin{proof}
From Lemma \ref{lem:randomUpdateBound} we have that for all $t\geq 1$,
\begin{eqnarray*}
\forall t\geq 1: \quad \E[h_{t+1}] \leq \left({1 - \frac{\eta_t}{2}}\right)\E[h_t] +  \frac{5\eta_t^2\beta\sqrt{\sqrt{2}\rank(\X^*)}}{2\alpha^{1/4}}\E[h_t]^{1/4} + \eta_t\xi_t.
\end{eqnarray*}

We are going to assume throughout the proof that $\xi_t \leq \E[h_t]/6$. It thus follows that
\begin{eqnarray}\label{eq:recurs2:1}
\forall t\geq 1: \quad \E[h_{t+1}] \leq \left({1 - \frac{\eta_t}{3}}\right)\E[h_t] +  \frac{5\eta_t^2\beta\sqrt{\sqrt{2}\rank(\X^*)}}{2\alpha^{1/4}}\E[h_t]^{1/4}.\end{eqnarray}

For all $t\geq 1$, define $v_t := \left({\frac{5\sqrt{\sqrt{2}\rank(\X^*)}\beta}{\alpha^{1/4}}}\right)^{-4/3}\E[h_t]$. Dividing both sides of Eq. \eqref{eq:recurs2:1} by $\left({\frac{5\sqrt{\sqrt{2}\rank(\X^*)}\beta}{\alpha^{1/4}}}\right)^{4/3}$, we have that
\begin{eqnarray}\label{eq:recurs2:2}
\forall t\geq 1: \quad v_{t+1} \leq \left({1 - \frac{\eta_t}{3}}\right)v_t+  \frac{\eta_t^2}{2}v_t^{1/4}.
\end{eqnarray}

We are going to prove by induction on $t$ that $v_t \leq \frac{C}{(t+t_0)^{4/3}}$ for suitable valus of $C, t_0$ and a sequence of step-sizes $\{\eta_t\}_{t\geq 1}$. Obviously for the base case $t=1$ to hold, we must restrict 
$\frac{C}{(t_0+1)^{4/3}} \geq v_1$.

Let us assume now that the induction hypothesis holds for some $t\geq 1$. 

Setting $\eta_t = \frac{C^{3/4}}{3(t+t_0)}$ in Eq. \eqref{eq:recurs2:2} we have that
\begin{eqnarray*}
v_{t+1} &\leq & v_t\left({1 - \frac{C^{3/4}}{9(t+t_0)}}\right) + \frac{C^{3/2}}{18(t+t_0)^2}v_t^{1/4}  \\
&\leq& \frac{C}{(t+t_0)^{4/3}}\left({1 - \frac{C^{3/4}}{9(t+t_0)}}\right) + \frac{C^{7/4}}{18(t+t_0)^{7/3}} \\
&=& \frac{C}{(t+t_0)^{4/3}}\left({1 - \frac{C^{3/4}}{18(t+t_0)}}\right) \\
&=&  \frac{C}{(t+1+t_0)^{4/3}}\left({1+\frac{1}{t+t_0}}\right)^{4/3}\left({1 - \frac{C^{3/4}}{18(t+t_0)} }\right) \\
&=&  \frac{C}{(t+1+t_0)^{4/3}}\left({1+\frac{1}{t+t_0}}\right)\left({1+\frac{1}{t+t_0}}\right)^{1/3}\left({1 - \frac{C^{3/4}}{18(t+t_0)} }\right)
\end{eqnarray*}

The single variable function $g(x) = x^{1/3}$ is concave on $(0,\infty)$, and thus, $g(1+x) \leq g(1) + g'(1)\cdot x = 1 + \frac{x}{3}$. Using this fact, we have that
\begin{eqnarray*}
v_{t+1} &\leq & \frac{C}{(t+1+t_0)^{4/3}}\left({1+\frac{1}{t+t_0}}\right)\left({1+\frac{1}{3(t+t_0)}}\right)\left({1 - \frac{C^{3/4}}{18(t+t_0)}}\right) \\
&<&  \frac{C}{(t+1+t_0)^{4/3}}\left({1+\frac{5}{3(t+t_0)}}\right)\left({1 - \frac{C^{3/4}}{18(t+t_0)}}\right) .
\end{eqnarray*}

Thus, choosing $C \geq (90/3)^{4/3}$ gives:
\begin{eqnarray*}
v_{t+1} \leq  \frac{C}{(t+1+t_0)^{4/3}}\left({1+ \frac{5}{3(t+t_0)}}\right)\left({1 - \frac{5}{3(t+t_0)}}\right) < \frac{C}{(t+1+t_0)^{4/3}},
\end{eqnarray*}

as needed.

We can now set values for $C, t_0$ under the constraints that 
\begin{eqnarray}\label{eq:recurs2:3}
i.\, C \geq 30^{4/3}, \qquad ii. \, \frac{C}{(t_0+1)^{4/3}} \geq v_1, \qquad iii. \, \forall t\geq 1: \,\eta_t = \frac{C^{3/4}}{3(t+t_0)}\in[0,2] .
\end{eqnarray}

As in the proof of Lemma \ref{lem:recurs1} it follows that constraining $t_0 \geq \frac{C^{3/4}}{6}-1$, will result in step-sizes that satisfy the conditions of Observation \ref{observ:largeCoeff}.

Moving to deal with the base case of the induction, again similarly to Lemma \ref{lem:recurs1}, we have that

\begin{eqnarray*}
v_1 &=& \left({\frac{\alpha^{1/4}}{5\beta\sqrt{\sqrt{2}\rank(\X^*)}}}\right)^{4/3}h_1 \leq \left({\frac{\alpha^{1/4}}{5\beta\sqrt{\sqrt{2}\rank(\X^*)}}}\right)^{4/3}\cdot 2\beta \\
&=& \left({\frac{\sqrt{2}\alpha^{1/4}}{5\beta^{1/4}\sqrt{\rank(\X^*)}}}\right)^{4/3} < 1,
\end{eqnarray*}
where the inequality follows since $\alpha \leq \beta$. Thus we must require that $\frac{C}{(1+t_0)^{4/3}} \geq 1$, which gives us the constraint $t_0 \leq C^{3/4}-1$.

Thus, the conditions in Eq. \eqref{eq:recurs2:3} boils down to the following constraints:
\begin{eqnarray*}
C \geq 30^{4/3}, \qquad  C^{3/4}-1 \geq t_0 \geq \frac{C^{3/4}}{6}-1 .
\end{eqnarray*}

For $C,t_0$ that indeed satisfy these constraints we can thus conclude that
\begin{eqnarray*}
\forall t\geq 1: \quad \E[h_t] \leq  \left({\frac{5\beta\sqrt{\rank(\X^*)}}{2^{3/4}\alpha^{1/4}}}\right)^{4/3}v_t 
\leq  \left({\frac{5C^{3/4}\beta\sqrt{\rank(\X^*)}}{2^{3/4}\alpha^{1/4}(t+t_0)}}\right)^{4/3} .
\end{eqnarray*}
\end{proof}

\begin{lemma}\label{lem:recurs3}
Let $C_,t_0$ be non-negative scalars that satisfy:
\begin{eqnarray*}
C \geq 2916, \qquad  C^{1/2}-1 \geq t_0 \geq \frac{C^{1/2}}{6}-1 .
\end{eqnarray*}
Then if for all $t\geq 1$ we define $\eta_t = \frac{C^{1/2}}{3(t+t_0)}$ and $\xi_0 = \beta$, $\forall t\geq 1: \, \xi_t = \frac{1}{6}\left({\frac{3\sqrt{2C}\beta}{\sqrt{\alpha}\lambda_{\min}(\X^*)(t+t_0)}}\right)^{2}$, it follows that all iterates of Algorithm \ref{alg:1} are feasible, and
\begin{eqnarray*}
\forall t\geq 1: \quad \E[h_t]\ \leq \left({\frac{3\sqrt{2C}\beta}{\sqrt{\alpha}\lambda_{\min}(\X^*)(t+t_0)}}\right)^{2}.
\end{eqnarray*}
\end{lemma}
                                                                                                                                                                             
\begin{proof}
From Lemma \ref{lem:randomUpdateBound} we have that for all $t\geq 1$,
\begin{eqnarray*}
\forall t\geq 1: \quad \E[h_{t+1}] \leq \left({1 - \frac{\eta_t}{2}}\right)\E[h_t] +  \frac{3\sqrt{2}\eta_t^2\beta}{2\sqrt{\alpha}\lambda_{\min}(\X^*)}\E[h_t]^{1/2} + \eta_t\xi_t.
\end{eqnarray*}
We are going to assume throughout the proof that $\xi_t \leq \E[h_t]/6$. It thus follows that
\begin{eqnarray}\label{eq:recurs3:1}
\forall t\geq 1: \quad \E[h_{t+1}] \leq \left({1 - \frac{\eta_t}{3}}\right)\E[h_t] +  \frac{3\sqrt{2}\eta_t^2\beta}{2\sqrt{\alpha}\lambda_{\min}(\X^*)}\E[h_t]^{1/2}.\end{eqnarray}
For all $t\geq 1$, define $v_t := \left({\frac{3\sqrt{2}\beta}{\sqrt{\alpha}\lambda_{\min}(\X^*)}}\right)^{-2}\E[h_t]$. Dividing both sides of Eq. \eqref{eq:recurs2:1} by $\left({\frac{3\sqrt{2}\beta}{\sqrt{\alpha}\lambda_{\min}(\X^*)}}\right)^{2}$, we have that
\begin{eqnarray}\label{eq:recurs3:2}
\forall t\geq 1: \quad v_{t+1} \leq \left({1 - \frac{\eta_t}{3}}\right)v_t+  \frac{\eta_t^2}{2}v_t^{1/2}.
\end{eqnarray}

We are going to prove by induction on $t$ that $v_t \leq \frac{C}{(t+t_0)^{2}}$ for suitable valus of $C, t_0$ and a sequence of step-sizes $\{\eta_t\}_{t\geq 1}$. Obviously for the base case $t=1$ to hold, we must restrict 
$\frac{C}{(t_0+1)^{2}} \geq v_1$.

Let us assume now that the induction hypothesis holds for some $t\geq 1$. 

Setting $\eta_t = \frac{C^{1/2}}{3(t+t_0)}$ in Eq. \eqref{eq:recurs2:2} we have that
\begin{eqnarray*}
v_{t+1} &\leq & v_t\left({1 - \frac{C^{1/2}}{9(t+t_0)}}\right) + \frac{C}{18(t+t_0)^2}v_t^{1/2} \\
&\leq &\frac{C}{(t+t_0)^{2}}\left({1 - \frac{C^{1/2}}{9(t+t_0)}}\right) + \frac{C^{3/2}}{18(t+t_0)^{3}} \\
&=&  \frac{C}{(t+1+t_0)^{2}}\left({1+\frac{1}{t+t_0}}\right)^{2}\left({1 - \frac{C^{1/2}}{18(t+t_0)} }\right) \\
&\leq&  \frac{C}{(t+1+t_0)^{2}}\left({1+\frac{3}{t+t_0}}\right)\left({1 - \frac{C^{1/2}}{18(t+t_0)} }\right) \\
\end{eqnarray*}

Thus, choosing $C \geq 2916$ gives:
\begin{eqnarray*}
v_{t+1} \leq  \frac{C}{(t+1+t_0)^{2}}\left({1+\frac{3}{t+t_0}}\right)\left({1 - \frac{3}{t+t_0} }\right) < \frac{C}{(t+1+t_0)^{2}},
\end{eqnarray*}

as needed.

We can now set values for $C, t_0$ under the constraints that 
\begin{eqnarray}\label{eq:recurs3:3}
i.\, C \geq 2916 \qquad ii. \, \frac{C}{(t_0+1)^{2}} \geq v_1, \qquad iii. \, \forall t\geq 1: \,\eta_t = \frac{C^{1/2}}{3(t+t_0)}\in[0,2] .
\end{eqnarray}

As in Lemma \ref{lem:recurs1}, it follows that in order for our step-sizes satisfy the conditions of Observation \ref{observ:largeCoeff}, we need to require that $t_0 \geq \frac{C^{1/2}}{6}-1$.

Also, for the base case of the induction, also similarly to Lemma \ref{lem:recurs1}, it holds that
\begin{eqnarray*}
v_1 \leq \left({\frac{\sqrt{\alpha}\lambda_{\min}(\X^*)}{3\sqrt{2}\beta}}\right)^{2}\cdot 2\beta = \left({\frac{\sqrt{2\alpha}\lambda_{\min}(\X^*)}{3\sqrt{2}\sqrt{\beta}}}\right)^{2} < 1
\end{eqnarray*}
where the second inequality follows since $\alpha \leq \beta$ and $\lambda_{\min}(\X^*) \leq 1$.
Thus in order to satisfy the constraint $v_1 \leq \frac{C}{(t_0+1)^2}$, it suffices to require $t_0 \leq \sqrt{C}-1$.

Thus, the conditions in Eq. \eqref{eq:recurs2:3} boils down to the following constraints:
\begin{eqnarray*}
C \geq 2916, \qquad  C^{1/2}-1 \geq t_0 \geq \frac{C^{1/2}}{6}-1 .
\end{eqnarray*}

For $C,t_0$ that indeed satisfy these constraints we can thus conclude that
We can thus conclude that
\begin{eqnarray*}
\forall t\geq 1: \quad \E[h_t] \leq   \left({\frac{3\sqrt{2}\beta}{\sqrt{\alpha}\lambda_{\min}(\X^*)}}\right)^{2}v_t \leq \left({\frac{3\sqrt{2C}\beta}{\sqrt{\alpha}\lambda_{\min}(\X^*)(t+t_0)}}\right)^{2}.
\end{eqnarray*}
\end{proof}


We can now finally wrap-up the proof of Theorem \ref{thm:main}.

\begin{proof}
The proof is an immediate consequence of Lemmas \ref{lem:recurs1}, \ref{lem:recurs2}, \ref{lem:recurs3}, and the observation that the step-size $\eta_t = \frac{54}{3(t+8)} = \frac{18}{t+8}$, which implicitly sets $t_0 = 8$  in all of above lemmas and corresponds to setting $C=54$ for Lemma \ref{lem:recurs1}, $C= 54^{4/3}$ in Lemma  \ref{lem:recurs2}, and $C=2916$ in Lemma \ref{lem:recurs3}, satisfies all lemmas together.
\end{proof}

\section{Preliminary Empirical Evaluation}\label{sec:experiments}

In this section we provide preliminary empirical evaluation of our approach. We evaluate our method, along with other conditional gradient variants, on the task of matrix completion. For a detailed presentation of the setting and the application of the conditional gradient method to this problem, we refer the reader to \cite{Jaggi10}.

\paragraph{Setting}
The underlying optimization problem for the matrix completion task is the following:
\begin{eqnarray}\label{eq:matComp}
\min_{\Z\in\ball_{d_1,d_2}(\theta)}\{f(\Z) := \frac{1}{2}\sum_{l=1}^n\left({\Z\bullet\matE_{i_l,j_l} - r_l}\right)^2\},
\end{eqnarray}
where $\matE_{i,j}$ is the indicator matrix for the entry $(i,j)$ in $\reals^{d_1\times d_2}$, and $\{(i_l,j_l,r_l)\}_{l=1}^n\subset[d_1]\times[d_2]\times\reals$. That is, our goal is to find a matrix with bounded nuclear norm (which serves as a convex surrogate for bounded rank) which matches best the partial observations given by $\{(i_l,j_l,r_l)\}_{l=1}^n$.

Since the feasible set is the nuclear ball, we use the reduction specified in Subsection \ref{sec:nuclear2spectra} to transform it to optimization over the spectrahedron.

The objective function in Eq. \eqref{eq:matComp} is known to have a smoothness parameter $\beta$ with respect to $\Vert\cdot\Vert_F$, which satisfies $\beta = O(1)$, see for instance \cite{Jaggi10}. While the objective function in Eq. \eqref{eq:matComp} is not strongly convex, it is known that under certain conditions, the matrix completion problem exhibit proprieties very similar to strong convexity, in the sense of Eq. \eqref{eq:strongconvexdist} (which is indeed the only consequence of strong convexity that we use in our analysis), known as \textit{restricted strong convexity} \cite{Negahban09}.

\paragraph{Two modifications of Algorithm \ref{alg:1}}
We implemented our rank-one-regularized conditional gradient variant, Algorithm \ref{alg:1} (denoted ROR-CG in our figures) with two modifications. First, on each iteration $t$, instead of picking an index $i_t$ of a rank-one matrix in the decomposition of the current iterate at random according to the distribution $(a_1,a_2,...,a_k)$, we choose it in a greedy way, i.e., we choose the rank-one component that has the largest product with the current gradient direction.
While this approach is computationally more expensive, it could be easily parallelized since all dot-product computations are independent of each other.
Second, after computing the eigenvector $\v_t$ using the choice of $\eta_t$ prescribed in Theorem \ref{thm:main}, we apply a line-search, as detailed in \cite{Jaggi10}, in order to the determine the optimal step-size given the direction $\v_t\v_t^{\top}-\x_{i_t}\x_{i_t}^{\top}$.

\paragraph{Baselines}
As baselines for comparison we used the standard conditional gradient method with exact line-search for setting the step-size (denoted CG in our figures)\cite{Jaggi10}, and the conditional gradient with away-steps variant, recently studied in \cite{Jaggi13c, Beck15,Jaggi15} (denoted Away-CG in our figures). While the away-steps variant was studied in the context of optimization over polyhedral sets, and its formal improved guarantees apply only in that setting, the concept of away-steps still makes sense for any convex feasible set. This variant also allows the incorporation of an exact line-search procedure to choose the optimal step-size.

\paragraph{Datasets}
We have experimented with two well known datasets for the matrix completion task: the \textsc{MovieLens100k} dataset for which $d_1 = 943, d_2 = 1682, n = 10^5$, and the \textsc{MovieLens1M} dataset for which $d_1 = 6040 , d_2 = 3952 , n \approx 10^6$. The \textsc{MovieLens1M} dataset was further sub-sampled to contain roughly half of the observations.  We have set the parameter $\theta$ in Problem \eqref{eq:matComp} to $\theta=10000$ for the \textsc{ML100k} dataset, and $\theta=30000$ for the  \textsc{ML1M} dataset.
\\

Figure \ref{fig:1} presents the objective \eqref{eq:matComp} vs. the number of iterations executed, for all methods, on both datasets. Each graph is the average over 5 independent experiments \footnote{We ran several experiments since the leading eigenvector computation in each one of the CG variants is randomized.}. 
It can be seen that our approach indeed improves significantly over the baselines in terms of convergence rate, for the setting under consideration.

\begin{figure}
    \centering
    \begin{subfigure}[b]{0.485\textwidth}
        \includegraphics[width=\textwidth]{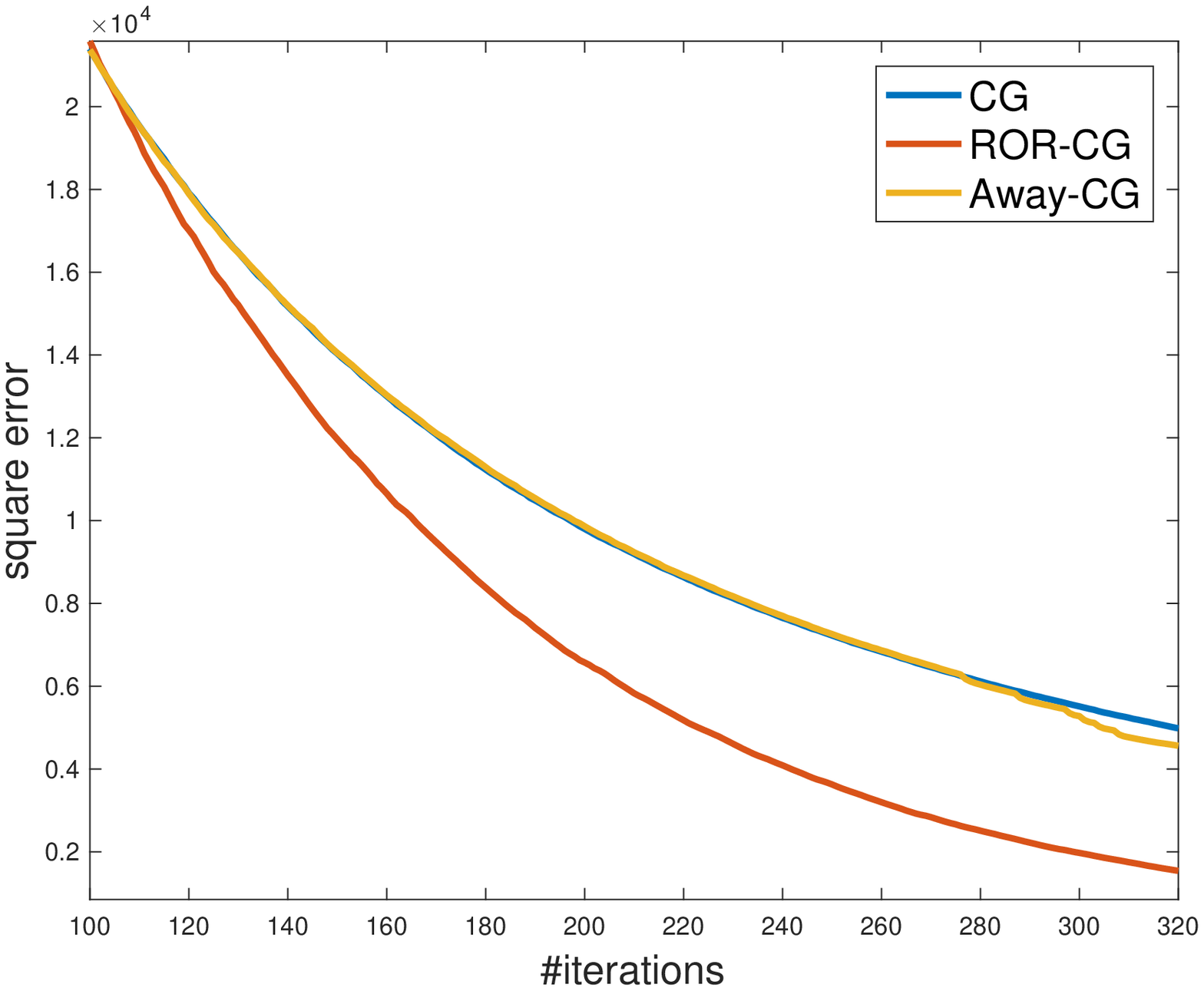}
    \end{subfigure}
    ~ 
    \begin{subfigure}[b]{0.485\textwidth}
        \includegraphics[width=\textwidth]{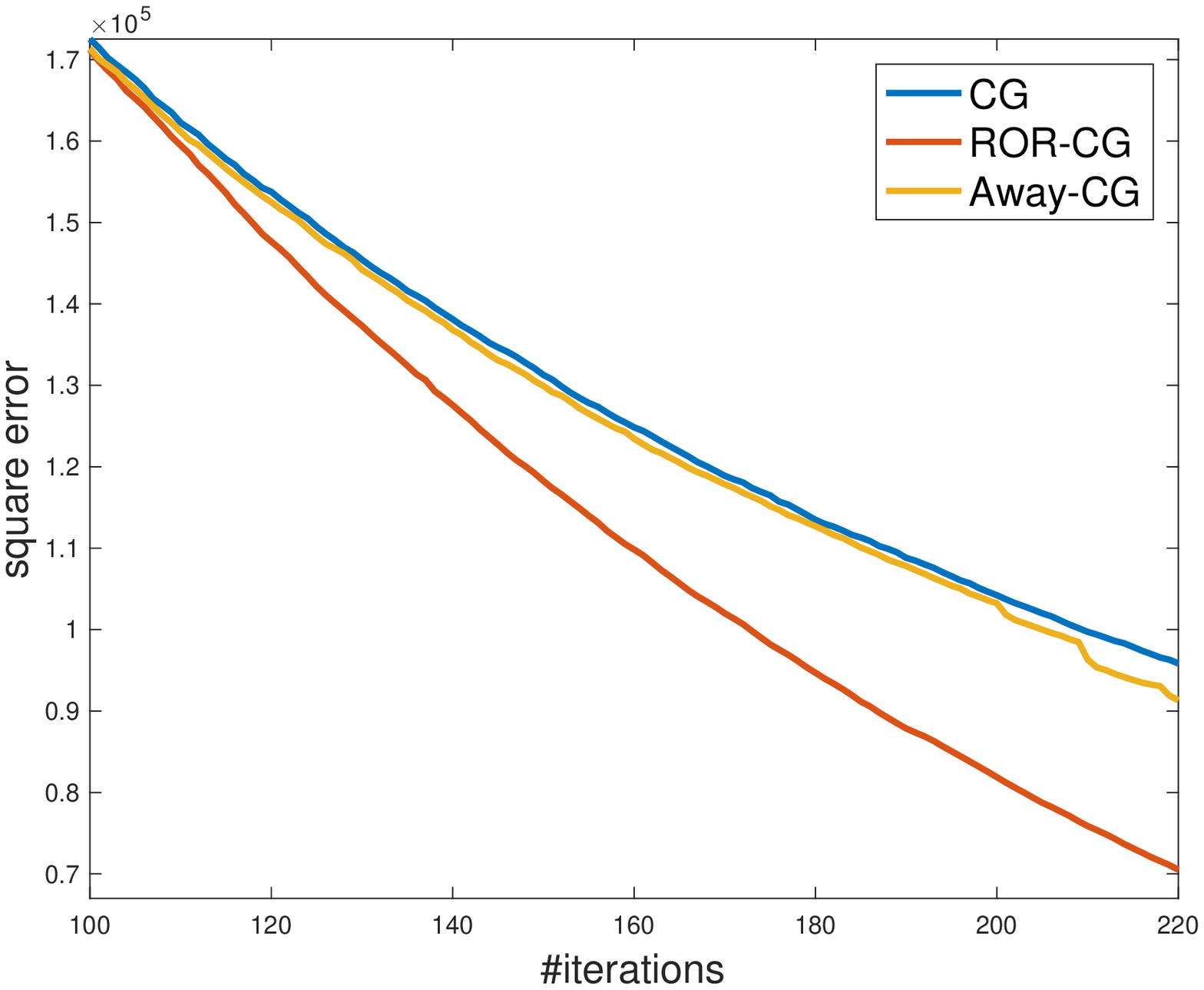}
    \end{subfigure}
    \caption{Comparison between conditional gradient variants for solving the matrix completion problem on the \textsc{MovieLens100k} (left) and  \textsc{MovieLens1M} (right) datasets.}\label{fig:1}
\end{figure}

\begin{figure}[h!]
  \centering

\end{figure}

\bibliographystyle{plain}
\bibliography{bib}

\appendix



\end{document}